\documentclass{mcom-l}

\usepackage[abbrev]{amsrefs}
\usepackage{amsmath}
\usepackage{amssymb}
\usepackage{geometry}
\usepackage{graphicx}
\usepackage{subcaption}
\usepackage{algorithm,algorithmic}
\usepackage{booktabs}
\usepackage{mathtools}

\usepackage[english]{babel}
\newcommand{\<}{\langle}
\renewcommand{\>}{\rangle}

\newcommand{\be}{\begin{equation}}
\newcommand{\ee}{\end{equation}}
\DeclareMathOperator{\supp}{supp}

\def\Chi{\raise .3ex \hbox{\large $\chi$}} 
\def\lsima{\hbox{\kern -.6em\raisebox{-1ex}{$~\stackrel{\textstyle<}{\sim}~$}}\kern -.4em}
\def\lsim{\hbox{\kern -.2em\raisebox{-1ex}{$~\stackrel{\textstyle<}{\sim}~$}}\kern -.2em}

\newcommand{\C}{\mathbb{C}}

\newcommand{\N}{\mathbb{N}}

\newcommand{\R}{\mathbb{R}}
\newcommand{\Z}{\mathbb{Z}}

\def\cB{{\mathcal B}}
\def\cC{{\mathcal C}}
\def\cD{{\mathcal D}}
\def\cE{{\mathcal E}}
\def\cF{{\mathcal F}}

\def\cM{{\mathcal M}}

\def\cO{{\mathcal O}}
\def\cP{{\mathcal P}}

\def\cS{{\mathcal{S}}}

\newcommand{\bh}{{\boldsymbol{h}}}
\newcommand{\bk}{{\boldsymbol{k} }}

\newcommand{\bi}{{\boldsymbol i}}
\newcommand{\bj}{{\boldsymbol j}}
\newcommand{\bone}{{\boldsymbol 1}}
\newcommand{\bzero}{{\boldsymbol 0}}
\newcommand{\kr}{{\boldsymbol e}}

\newcommand{\bl}{{\boldsymbol l}}

\newcommand{\bt}{{\boldsymbol t }}

\newcommand{\bw}{{\boldsymbol w}}
\newcommand{\bx}{{\boldsymbol x}}
\newcommand{\by}{{\boldsymbol y}}
\newcommand{\bz}{{\boldsymbol z}}

\newtheorem{theorem}{Theorem}[section]
\newtheorem{lemma}[theorem]{Lemma}

\theoremstyle{definition}
\newtheorem{definition}[theorem]{Definition}

\theoremstyle{remark}
\newtheorem{remark}[theorem]{Remark}
\newtheorem{proposition}[theorem]{Proposition}

\numberwithin{equation}{section}

\begin{document}

\title{Reconstruction Plans and Efficient Rank-1 Lattice Construction for Chebyshev Expansions Over Lower Sets}

\author{Abdelqoddous Moussa}\thanks{College of Computing, University Mohammed VI Polytechnic, Morocco. \texttt{abdelqoddous.moussa@um6p.ma}. }

\author{Moulay Abdellah Chkifa} \thanks{  \texttt{moulay-abdellah.chkifa@polytechnique.org}}






\begin{abstract}
    This study focuses on constructing efficient rank-1 lattices that enable the exact integration and reconstruction of functions within Chebyshev spaces, based on finite lower index sets. We establish the equivalence of different reconstruction plans under specific conditions for certain lower sets. Furthermore, we introduce a   heuristic component-by-component (CBC) algorithm that efficiently identifies admissible generating vectors and suitable numbers of nodes $n$, optimizing both memory usage and computational time.
\end{abstract}

\maketitle

\section{Introduction}
\label{intro}
Function integration and reconstruction are fundamental problems in computational mathematics, with applications ranging from numerical analysis and machine learning to engineering and the physical sciences. Achieving accurate and efficient approximations, especially in high-dimensional settings, is challenging due to the exponential growth of computational resources required—a phenomenon often referred to as the curse of dimensionality \cite{caflisch1998monte}. Rank-1 lattices have emerged as a powerful framework to address these challenges, offering structured sampling schemes that facilitate efficient and accurate approximation of multivariate functions.

A rank-1 lattice rule approximates the $d$-dimensional integral over the unit cube $[0,1]^d$ by evaluating the integrand at a carefully chosen set of lattice points:
\begin{equation}
    I_d(f) = \int_{[0,1]^d} f(\bx) \, \mathrm{d}\bx \approx Q_{n,\bz}(f) = \frac{1}{n} \sum_{i=0}^{n-1} f\left(\frac{i \bz}{n} \mod 1\right),
\end{equation}
where $\bz \in \mathbb{Z}^d$ is the generating vector, and $n$ is the number of nodes. The selection of $\bz$ and $n$ is critical for ensuring the accuracy and efficiency of the approximation. Optimal rank-1 lattices minimize the number of nodes $n$ while satisfying exactness conditions for specific function spaces, thereby balancing computational resources with approximation fidelity.

Previous works have extensively explored algorithms for constructing rank-1 lattices tailored to various function spaces. For instance, \cite{cools2010constructing} introduced weighted degrees of exactness for anisotropic integrands and presented CBC strategies for their construction. Subsequent works, such as \cite{Kammerer2020AFP}, advanced these methodologies by developing deterministic algorithms for multiple rank-1 lattices, enhancing their applicability in high-dimensional approximation tasks. Additionally, \cite{gross2021deterministic} extended CBC approaches to deterministic settings, focusing on exact integration within truncated Fourier spaces.

Chebyshev spaces, characterized by their optimal convergence properties and prevalence in spectral methods~\cite{Chkifa2016PolynomialAV}, present unique opportunities and challenges for rank-1 lattice constructions. Unlike Fourier spaces, Chebyshev spaces require handling polynomial bases with specific orthogonality properties, necessitating tailored reconstruction strategies. Lower index sets underpin these function spaces by ensuring that all necessary lower-order terms are included for accurate function representation~\cite{ChkifaLower}.

The primary contributions of this paper are threefold:

\begin{itemize}
    \item \textbf{Theoretical Lower Bounds:} We establish lower bounds on the minimal number of nodes $n$ required for exact reconstruction in Chebyshev spaces, and derive closed-form expressions for specific classes of lower index sets.

    \item \textbf{Equivalence of Reconstruction Plans:} We investigate different reconstruction strategies—denoted as Plans A, B, and C—and establish conditions under which these plans are equivalent for certain lower sets. This equivalence allows for flexibility in algorithm design, enabling the selection of the most computationally efficient plan without compromising reconstruction accuracy.
    
    \item \textbf{Heuristic CBC Algorithm:} We introduce a   heuristic component-by-component (CBC) algorithm specifically designed for lower sets in Chebyshev spaces. This algorithm leverages the structural properties of lower sets to efficiently generate optimal generating vectors and determine suitable numbers of nodes $n$, significantly reducing both memory usage and computational time compared to exhaustive search methods.
    
\end{itemize}


The paper is organized as follows. Section~\ref{sec:pre} introduces notation, cubature rules, rank-1 lattice properties and lower sets. Section~\ref{sec:Necessary conditions} provides lower bounds on nodes for reconstruction plans in Chebyshev expansions. Section~\ref{sec:equivalence_plans_ABC} presents our main theoretical findings, including equivalence results for reconstruction plans. Section~\ref{sec:specialized_cbc} describes the proposed heuristic CBC algorithms, their complexity, and implementation details. Section~\ref{sec:experiments} showcases numerical experiments comparing the algorithms' efficiency and accuracy. Section~\ref{sec:conclusion} summarizes our findings and proposes future research directions. Appendix~\ref{appendix:lower sets} details some properties of lower sets, including their operations and characterizations of the so-called maximal and extremal elements.
A cardinality comparison result of broad interest is given in \ref{propSumInequality}.

\section{Preliminaries}
\label{sec:pre}


We consider the following classical spaces
\be
\cF^{\rm Four}:= \Big\{ f \in L^2\;|\; f: [0,1]^d \mapsto \C, 
f(x)=\sum_{\bk \in \Z^d} {\widehat f}_\bk e_\bk(\bx),~ 
\sum_{\bk \in \Z^d} |{\widehat f}_\bk| < \infty  \Big\},
\ee
\be
\cF^{\rm cos}:= \Big\{ f \in L^2\;|\; f: [0,1]^d \mapsto \R, 
f(x)=\sum_{\bk \in \N_0^d} {\widehat f}_\bk \phi_\bk(\bx),~\sum_{\bk \in \N_0^d} |{\widehat f}_\bk| < \infty  \Big\},
\ee
\be
\cF^{\rm Cheb}:= \Big\{ f \in L^2\;|\; f: [-1,1]^d \mapsto \R, 
f(x)=\sum_{\bk \in \N_0^d} {\widehat f}_\bk \eta_\bk(\bx),~
\sum_{\bk \in \N_0^d} |{\widehat f}_\bk| < \infty  \Big\}.
\ee
The underlying measure is the Lebesgue measure $d\bx$ for the first
two spaces and the tensor product arcsine measure
$\prod_{j=1}^d \frac{dx_j}{\pi\sqrt{1-x_j^2}}$ for the third space. 
For each space, the basis used in the expansion is a complete 
orthonormal system. They all have a tensor product structure
$e_\bk = \otimes_{j=1}^d e_{k_j}$,
$\phi_\bk = \otimes_{j=1}^d \phi_{k_j}$,
$\eta_\bk = \otimes_{j=1}^d \eta_{k_j}$,
where $e_0\equiv1$, $\phi_0\equiv1$, $\eta_0\equiv1$ 
otherwise $e_{k}(x) = e^{2i\pi k x}$, 
$\phi_{k}(x) = \sqrt2 \cos (\pi k x)$ and $\eta_{k}(x) = \sqrt2 T_k (x)$, where $T_k$ is the 
$k$-th Chebyshev polynomial defined by  the identity: 
$$
T_k(x) = \cos(k \arccos(x)), \quad k = 0, 1, 2, ..., \quad \text{for } x \in [-1, 1].
$$
We recall the following 
\begin{itemize}
\item $(e_k)_{k\in\Z}$ is a complete orthonormal 
system for $L^{2} ([0,1],dx)$,
\item $(\phi_{k})_{k\in\N_0}$ is a complete 
orthonormal system for $L^2([0,1],dx)$
\item $(\eta_{k})_{k\in\N_0}$ is a complete 
orthonormal system for $L^2([-1,1],dx/{[\pi\sqrt{1-x^2]}})$.
\end{itemize}

In our paper, we designate the first space $\cF^{\mathrm{Four}}$ as the Fourier space, representing periodic functions over the unit cube $[0,1]^d$ characterized by absolutely converging Fourier series. The second space, denoted $\cF^{\mathrm{cos}}$, is termed the cosine space, of periodic functions over $[0,1]^d$ with absolutely converging cosine series. The third space, 
$\cF^{\mathrm{Cheb}}$, denotes the space of functions possessing absolutely converging Chebyshev series.

Given an arbitrary finite set $\Lambda$ of indices in $\mathbb{Z}^d$ 
(respectively in $\mathbb{N}_0^d$), the subspace $\cF_\Lambda^{\rm Four}$ 
(respectively $\cF_\Lambda^{\rm cos}$, $\cF_\Lambda^{\rm Cheb}$) is 
the space of all functions whose Fourier (respectively cosine, Chebyshev) 
series is supported solely on $\Lambda$,  i.e., $\widehat{f}_\bk=0$
for $\bk \notin \Lambda$ \footnote{Cosine and Chebyshev space are equivalent,
$ f \in  \cF_\Lambda^{\rm Cheb} \iff f_{\rm \cos} :=f(\cos (\pi \boldsymbol{\cdot})) \in \cF_\Lambda^{\rm \cos}$}.
For instance
\be
f \in \cF_{\Lambda}^{\rm Four }: \quad 
f(x)=\sum_{\bk \in \Lambda} 
\widehat{f}_\bk e^{2 \pi \mathrm{i} \bk \cdot \bx}.
\ee
\subsection{Approximation requirements}

For the sake of clarity, in this section we will employ a unified 
notation for the spaces and bases we have introduced. We use notation 
$(\alpha_{\bk})_{\bk}$ for the multi-variate complete orthonormal 
systems and use conjugate notation $\overline {\alpha_\bk}$. 
The integral $I(f)$ and inner product $\<f,g\> =I(f\overline{g})$ 
are understood from the context. We will also use $\cF_\Lambda$ 
for truncated spaces, without specifying Fourier, cosine or Chebyshev.

Placing ourselves in one of the three spaces, a cubature (or quadrature) 
is simply defined by $Q_n (f) = \sum_{j=1}^n w_j f(\bt_i)$ for given 
weights $w_j >0$ and given nodes $\bt_i$ belonging to the underlying domain.
One can be interested in the following properties  
\begin{itemize}
\item {\bf Integral exactness}. To demand that the cubature is 
exact for all functions which are supported solely on $\Lambda$, i.e., 
$Q_n(f) = I(f)$ for all $f \in \cF_\Lambda$. This holds if and only if
\be
Q_n(\alpha_\bk) = I(\alpha_\bk) = \delta_{\bk,\bzero},
\quad\quad\quad \bk \in \Lambda,
\ee
where $\delta$ is the Kronecker symbol, i.e. the cubature formula integrates exactly all basis functions
$\alpha_\bk$ for $\bk \in \Lambda$.

\item {\bf Function reconstruction}. To reconstruct all functions supported on a given lattice $\Lambda$, we define the approximation operator using $Q_n$ as
\(
f^a({\bx}) = \sum_{\bk \in \Lambda} \widehat{f}_\bk^a \, \alpha_\bk({\bx}),
\)
where the coefficients $\widehat{f}_\bk^a$ are given by the cubature formula
$
\widehat{f}_\bk^a = Q_n \left(f \overline{\alpha_\bk}\right),
\)
and approximate the exact coefficients
$
\widehat{f}_\bk = \langle f, \alpha_\bk \rangle = I\left(f \overline{\alpha_\bk}\right),
\)
where $I$ represents the integral operator and $\langle \cdot, \cdot \rangle$ denotes the inner product.

To ensure accurate reconstruction, we impose the \textit{non-aliasing condition}:
$
\widehat{f}_\bk^a = \widehat{f}_\bk, \quad \forall \, \bk \in \Lambda \text{ and } f \in \mathcal{F}_\Lambda,
\)
where $\mathcal{F}_\Lambda$ is the space of functions supported solely on the lattice $\Lambda$. This condition ensures that the operator $f \in \mathcal{F} \mapsto f^a \in \mathcal{F}_\Lambda$ is a valid reconstruction operator. If contributions from $\widehat{f}_{\bk'}$ for $\bk' \notin \Lambda$ affect $\widehat{f}_\bk^a$, this introduces \textit{aliasing}.

The non-aliasing condition holds if and only if the cubature formula satisfies:
\be
Q_n (\alpha_\bk \overline {\alpha_{\bk'}}) =
\delta_{\bk, \bk'} =\< \alpha_\bk, \alpha_{\bk'} \>,
\quad\quad\quad \bk, \bk' \in \Lambda.
\label{condition_reconstruction}
\ee
where $\delta_{\bk, \bk'}$ is the Kronecker delta. This ensures the orthogonality of the basis functions $\alpha_\bk$, eliminating aliasing errors.

\item {\bf Function reconstruction and bi-orthonormality}. 
Reconstruction using cubature formulas can be approached through the framework of bi-orthonormality. Let $(\beta_\bk)_{\bk \in \Lambda}$ be a set of functions in $\mathcal{F}_\Lambda$ that is bi-orthonormal to $(\alpha_\bk)_{\bk \in \Lambda}$, satisfying:
\(
\langle \alpha_\bk, \beta_{\bk'} \rangle = \delta_{\bk, \bk'}, \quad \forall \, \bk, \bk' \in \Lambda,
\)
where $\langle \cdot, \cdot \rangle$ is the inner product and $\delta_{\bk, \bk'}$ is the Kronecker delta.

Given this bi-orthonormal basis, we define the approximation:
\be
f^b({\bx}) = \sum_{\bk \in \Lambda} \widehat{f}_\bk^b \, \alpha_\bk({\bx}),
\quad
\widehat{f}_\bk^b = \frac{Q_n \left(f \overline{\beta_{\bk}}\right)}{c_\bk},
\ee
where $Q_n$ is a cubature rule, $c_\bk \neq 0$ are nonzero coefficients, and $\widehat{f}_\bk^b$ are determined using $\beta_\bk$ instead of $\alpha_\bk$. This scheme reproduces $\mathcal{F}_\Lambda$—or equivalently the basis functions $\alpha_\bk$—if and only if:
\be
Q_n \left(\alpha_\bk \overline{\beta_{\bk'}}\right) =
c_\bk \delta_{\bk, \bk'},
\quad \forall \, \bk, \bk' \in \Lambda.
\label{bi-orthonormality}
\ee

The reproduction criterion in \eqref{bi-orthonormality} aligns with the duality framework. For integral exactness, one seeks a linear form $\alpha_\mathbf{0}^* : f \mapsto \alpha_\mathbf{0}^*(f)$ satisfying:
\(
\alpha_\mathbf{0}^*(\alpha_\bk) = \delta_{\mathbf{0}, \bk}, \quad \forall \, \bk \in \Lambda,
\)
which guarantees $\alpha_\mathbf{0}(f) = \widehat{f}_\mathbf{0} = I(f)$ for any $f \in \mathcal{F}_\Lambda$. Similarly, reconstruction requires a family of linear forms $\alpha_\bh^* : f \mapsto \alpha_\bh^*(f)$ indexed by $\Lambda$, such that:
\(
\alpha_\bh^*(\alpha_\bk) = \delta_{\bh, \bk}, \quad \forall \, \bh, \bk \in \Lambda.
\)
This family ensures that the operator $\sum_{\bh \in \Lambda} \alpha_\bh^*(f) \alpha_\bh$ reconstructs all functions in $\mathcal{F}_\Lambda$.

The forms $\alpha_\bh^*$ may not depend explicitly on $\Lambda$. Notable examples include $f \mapsto \widehat{f}_\bh^\square$ for $\square \in \{a, b, c\}$, as introduced in \cite{rank1}. These approaches provide robust tools for reconstruction in high-dimensional settings.
\end{itemize}
\subsection{Rank-1 lattices rules}
Rank-1 lattices are extensively studied for integration, reconstruction, and approximation in the Fourier space. Given a generating vector $\bz \in \mathbb{Z}^d$, the $n$ points of a rank-1 lattice are defined as:
\be
\mathbf{t}_i = \frac{i \bz \bmod n}{n} \in [0,1[^d, \quad \text{for } i = 0, \ldots, n-1.
\ee
For a function \(f\) in Fourier space (or other spaces such as Cosine or Chebyshev), the cubature rules can be expressed as:  
\be  
\frac{1}{n} \sum_{i=0}^{n-1} f\left(\bt_i\right), \quad\quad  
\frac{1}{n} \sum_{i=0}^{n-1} f\left(\varphi_{\text{tent}}(\bt_i)\right), \quad \quad 
\frac{1}{n} \sum_{i=0}^{n-1} f\left(\cos\left(2 \pi \bt_i\right)\right),  
\ee  
where the first formula corresponds to a standard rank-1 lattice rule, the second to a tent-transformed lattice rule, and the third to a tent-transformed then cosine-transformed lattice rule. The tent transformation \(\varphi_{\text{tent}}: [0, 1]^d \to [0, 1]^d\) is defined as  
\( 
\varphi_{\text{tent}}(\mathbf{t}) = (1-|2t_1-1|,\dots,1-|2t_d-1|)
\).
In view of the symmetry rank-1 lattices i.e. ($\bt_{n-i} = \bone - \bt_{i}$ $\forall i\in\{1,\dots,n-1\}$), 
and that of tent and cosine functions, only half the points are needed for transformed rules
 (see \cites{Cools2016TenttransformedLR,rank1}). More precisely, at most 
 $\lfloor n/2+1\rfloor$ function evaluations are needed. 

The first formula corresponds to a rank-1 lattice rule, which is used to approximate the integral
\[
I(f) = \int_{[0,1]^d} f({\bx}) \, \mathrm{d} {\bx}, \quad \text{for } f \in \mathcal{F}^{\text{Four}}.
\]

Rank-1 lattices exhibit an important property, known as the \textit{character property}, which states that the cubature sum of exponential basis functions can only take the values 1 or 0. Specifically:
\be
\frac{1}{n} \sum_{i=0}^{n-1} e^{\mathrm{i} 2 \pi i \, \bh \cdot \bz / n} = 
\begin{cases}
1 & \text{if } \bh \cdot \bz \equiv_n 0, \\
0 & \text{otherwise}.
\end{cases}
\ee

This property, combined with the relation $e_\bk \overline{e_{\bk'}} = e_{\bk - \bk'}$, enables a reformulation of the integral exactness and reconstruction requirements discussed earlier.
\begin{lemma}[See \cites{Kmmerer2014ReconstructingMT,Potts2016SparseHF,rank1}] 
Let $\Lambda \subset \mathbb{Z}^d$ be an arbitrary index set. A lattice rank-1 rule $Q_n^*$ with $n$ points and generating vector $\bz$:
\begin{itemize}
    \item integrates exactly all functions $f \in \mathcal{F}^{\rm Four}_\Lambda$ if and only if
    \be
    \bh \cdot \bz \not\equiv_n 0 \quad \text{for all} \quad \bh \in \Lambda \setminus \{\mathbf{0}\},
    \ee
    \item reconstructs exactly the Fourier coefficients of all $f \in \mathcal{F}^{\rm Four}_\Lambda$ by ${\widehat f^a} = Q_n^*(f \overline{e_{\bk}})$ if and only if
    \be
    \bh \cdot \bz \not\equiv_n \bh' \cdot \bz \quad \text{for all} \quad 
    \bh \neq \bh'  \in \Lambda.
    \ee
\end{itemize}
\label{lemmaFourier}
\end{lemma}

Requirements for integral exactness and reconstruction, analogous to Lemma \ref{lemmaFourier}, hold for the spaces $\mathcal{F}_\Lambda^{\rm cos}$ and $\mathcal{F}_\Lambda^{\rm Cheb}$. To this end, let $\mathcal{M}(\Lambda)$ denote the mirrored set of $\Lambda$, defined as:
\[
\mathcal{M}(\Lambda) := \left\{ \sigma(\bk) : \bk \in \Lambda, \, \sigma \in \mathcal{S}_{\bk} \right\},
\]
where $\mathcal{S}_{\bk}$ is the set of $2^{|\bk|_0}$ component-wise sign flips applied to $\bk$. 

\begin{lemma}[\cite{rank1}]
Let $\Lambda \subset \mathbb{N}_0^d$ be an arbitrary index set. A tent-transformed and then cosine transformed lattice rule of a rank-1 lattice rule $Q_n^*$ with $n$ points and generating vector $\bz$ integrates exactly all functions $f \in \mathcal{F}^{\rm Cheb}_\Lambda$ if and only if:
\be
\bh \cdot \bz \not\equiv_n 0 \quad \text{for all} \quad \bh \in 
\mathcal{M}(\Lambda) \setminus \{\mathbf{0}\}.
\label{requirePlan0}
\ee
\label{lemmaCosineExactness}
\end{lemma}

To rigorously present the reconstruction results established in \cite{rank1}, we briefly describe the cubature formulas introduced therein. Given a rank-1 lattice rule $Q_n^*$ with $n$ points and generating vector $\bz$, we define the following:
\begin{align}
{\widehat f}_\bk^a
&:= Q_n^* \big( f(\cos(2\pi \boldsymbol{\cdot})) ~ \eta_\bk (\cos(2\pi \boldsymbol{\cdot})) \big), \\
{\widehat f}_\bk^b
&:= Q_n^* \big( f(\cos(2\pi \boldsymbol{\cdot}))~ \sqrt{2}^{|\bk|_0} \cos (2 \pi \bk \cdot \boldsymbol{\cdot}) \big), \\
{\widehat f}_\bk^c
&:= \frac{Q_n^* \big( f(\cos(2\pi \boldsymbol{\cdot}))~ \sqrt{2}^{|\bk|_0} \cos (2 \pi \bk \cdot \boldsymbol{\cdot}) \big)}{c_{\bk}},
\end{align}
where $c_{\bk} := \# \{ \sigma \in \mathcal{S}_{\bk} : \sigma(\bk) \cdot \bz \equiv_n \bk \cdot \bz \}$.

The reconstruction requirement based on ${\widehat f}_\bk^a$ (Plan A) is relatively straightforward. By observing:
\be
\int_{[0,1]^d} 
\phi_{\bk'} ( \varphi_{\text{tent}} ({\bx})) 
\sqrt{2}^{|\bk|_0} \cos (2 \pi \bk \cdot {\bx}) \, \mathrm{d}{\bx}
=\delta_{\bk',\bk}, \quad \bk',\bk \in \mathbb{N}_0^d,
\ee
one can use functions $\beta_\bk = \sqrt{2}^{|\bk|_0} \cos (2 \pi \bk \cdot {\bx})$ to 
perform the bi-orthonormality approach for cosine spaces, or equivalently for
Chebyshev spaces.

\begin{lemma}[\cite{rank1}]
Let $\Lambda \subset \mathbb{N}_0^d$ be an arbitrary index set. The rank-1 lattice rule $Q_n^*$ with $n$ points and generating vector $\bz$ reconstructs exactly the Chebyshev coefficients of all $f \in \mathcal{F}^{\rm Cheb}_\Lambda$:
\begin{itemize}
    \item{ {Plan A:}} by ${\widehat f} = {\widehat f^a}$ if and only if
    \be
    \bh \cdot \bz \not\equiv_n \bh' \cdot \bz \quad \text{for all} \quad 
    \bh \neq \bh' \in \mathcal{M}(\Lambda),
    \label{requirePlanA}
    \ee
    \item {Plan B:} by ${\widehat f} = {\widehat f^b}$ if and only if
    \be
    \sigma(\bh) \cdot \bz \not\equiv_n \bh' \cdot \bz \quad \text{for all} \quad 
    \bh, \bh' \in \Lambda,~\sigma \in \mathcal{S}_{\bh},~{\rm s.t.}~ \sigma (\bh) \neq \bh',
    \label{requirePlanB}
    \ee
    \item {Plan C:} by ${\widehat f} = {\widehat f^c}$ if and only if
    \be
    \sigma(\bh) \cdot \bz \not\equiv_n \bh' \cdot \bz \quad \text{for all} \quad 
    \bh, \bh' \in \Lambda,~ \sigma \in \mathcal{S}_{\bh},~{\rm s.t.}~ \bh \neq \bh'.
    \label{requirePlanC}
    \ee
\end{itemize}
\label{lemmaCosineReconstruction}
\end{lemma}

It is evident that condition \eqref{requirePlanA} is stricter than \eqref{requirePlanB}. For any $\bh, \bh' \in \Lambda$ ($\Lambda \subset \mathbb{N}_0^d$) and $\sigma \in \mathcal{S}_{\bh}$, it is clear that $\sigma (\bh) \neq \bh'$ unless $\bh = \bh'$ and $\sigma = \text{id}$. Consequently, satisfying \eqref{requirePlanB} is equivalent to satisfying \eqref{requirePlanC} with the additional requirement that
\(
\sigma(\bh) \cdot \bz \not\equiv_n \bh \cdot \bz, \quad \forall \bh \in \Lambda,~\sigma \in \mathcal{S}_{\bh}^*,
\)
where $\mathcal{S}_{\bh}^* = \mathcal{S}_{\bh} - \{\text{id}\}$. This means that in \eqref{requirePlanC}, it is permissible for $\bh \cdot \bz ~(\text{mod}~ n)$ to appear in the set 
$\{\sigma(\bh) \cdot \bz ~(\text{mod}~ n)\}_{\sigma \in \mathcal{S}_{\bh}^*}$.

To formalize admissibility across the different plans, we introduce a unified notation. For an 
arbitrary index set $\Lambda \subset \mathbb{N}_0^d$, let $\mathcal{P}_0(\Lambda)$ denote 
the set of pairs $(n, \bz) \in \mathbb{N} \times \mathbb{Z}^d$ that satisfy \eqref{requirePlan0}. 
Similarly, we define
\(
\mathcal{P}_{A}(\Lambda), \quad \mathcal{P}_{B}(\Lambda), \quad \mathcal{P}_{C}(\Lambda),
\)
as the sets of pairs $(n, \bz) \in \mathbb{N} \times \mathbb{Z}^d$ satisfying \eqref{requirePlanA}, \eqref{requirePlanB}, and \eqref{requirePlanC}, respectively. We say that a pair $(n, \bz)$ is \textit{admissible for plan} $\square$ if $(n, \bz) \in \mathcal{P}_{\square}(\Lambda)$, where $\square \in \{0, A, B, C\}$.

From the preceding discussion, it follows that
\(
\mathcal{P}_A(\Lambda) \subset \mathcal{P}_B(\Lambda) \subset \mathcal{P}_C(\Lambda).
\)

To provide further clarity, we summarize the incremental requirements for plans A, B, and C in Table~\ref{tab:IncrementalABC}. Here, O-O refers to the one-to-one condition $\bh \mapsto \bh \cdot \bz$.

\begin{table}[h!]
\centering
\begin{tabular}{| l l |c|c|c|c|}
\hline
Requirement  & & O-O & Plan C &  Plan B & Plan A \\ \hline
$\bh' \cdot \bz \not\equiv_n \bh \cdot \bz$ & $\bh' \neq \bh  \in \Lambda$ 
& \checkmark & \checkmark & \checkmark & \checkmark \\
$\bh' \cdot \bz \not\equiv_n \sigma(\bh) \cdot \bz$ & $\bh' \neq \bh \in \Lambda,~\sigma \in \mathcal{S}_{\bh}^*$
&  & \checkmark & \checkmark & \checkmark  \\
$\bh \cdot \bz \not\equiv_n \sigma(\bh) \cdot \bz$ & $\bh  \in \Lambda,~\sigma \in \mathcal{S}_{\bh}^*$ 
&  &  & \checkmark & \checkmark \\
$\sigma'(\bh') \cdot \bz \not\equiv_n \sigma(\bh) \cdot \bz$ &  $\bh' \neq \bh  \in \Lambda,~\sigma \in \mathcal{S}_{\bh}^*,~\sigma' \in \mathcal{S}_{\bh'}^*$ 
&  &&&\checkmark  \\  \hline    
\end{tabular}
\caption{Incremental differences between requirements O-O and plans C/B/A.}
\label{tab:IncrementalABC}
\end{table}

For a given $\bz \in \mathbb{Z}^d$, let $n^*_0(\Lambda, \bz)$ be the smallest integer $n$ such that $(n, \bz) \in \mathcal{P}_0(\Lambda)$. If no such $n$ exists (e.g., $\bz = \mathbf{0}$ and $\Lambda \neq \{\mathbf{0}\}$), we conventionally set $n^*_0(\Lambda, \bz) = +\infty$. Similarly, we define $n^*_A(\Lambda, \bz)$, $n^*_B(\Lambda, \bz)$, and $n^*_C(\Lambda, \bz)$ for the other plans. 

Additionally, let $n^*_0(\Lambda)$ denote the smallest $n$ such that $(n, \bz) \in \mathcal{P}_0(\Lambda)$ for some $\bz$. We define $n^*_A(\Lambda)$, $n^*_B(\Lambda)$, and $n^*_C(\Lambda)$ analogously.

Upper bounds for $n^*_{\square}(\Lambda)$, denoted by $p_{\square}^*(\Lambda)$, are established in \cite{rank1}:
\be
\label{eq: upper bounds}
\begin{aligned}
p_0^*(\Lambda) &= \text{smallest prime larger than } 
\|\Lambda\|_\infty \text{ and than } 
\frac{\#(\mathcal{M}(\Lambda) \setminus \{\mathbf{0}\})}2 + 1, \\   
p_A^*(\Lambda) &= \text{smallest prime larger than }
2\|\Lambda\|_\infty \text{ and than } 
 \frac {\#(\mathcal{M}(\Lambda) \oplus \mathcal{M}(\Lambda)) + 1}2, \\
p_B^*(\Lambda) &= \text{smallest prime larger than } 
2\|\Lambda\|_\infty \text{ and than } 
\#(\Lambda \oplus \mathcal{M}(\Lambda)),\\
p_C^*(\Lambda) &= \text{smallest prime larger than }  
2\|\Lambda\|_\infty \text{ and than } 
\#\Lambda\cdot \#\mathcal{M}(\Lambda),
\end{aligned}
\ee
where $\|\Lambda\|_\infty := \max \{\|\bk\|_{\infty} : \bk \in \Lambda\}$, and \(
\Lambda \oplus \Lambda' \;=\; \{\bk + \bk' \,:\, \bk\in \Lambda,\,\bk'\in \Lambda'\},\) the \textit{sum} of two sets $\Lambda,\Lambda'\subset \Z^d$. We note that for lower sets, the quantities
$\|\Lambda\|_\infty$ and $2\|\Lambda\|_\infty$ are inherently smaller than 
the cardinalities they compare against.

\subsection{Lower sets}
\label{sec 5}
\begin{definition}
    An index set \( \Lambda \subset \mathbb{N}_0^d \) is called \textit{lower}\footnote{also called downward closed} if it satisfies the following property:
    \be
    \bk \in \Lambda \quad \text{and} \quad \bh \leq \bk \quad \implies \quad \bh \in \Lambda.
    \ee
    Here, the comparison \( \bh \leq \bk \) is understood component-wise, meaning \( h_i \leq k_i \) for all \( 1 \leq i \leq d \).
\end{definition}

A lower set \( \Lambda \) can also be characterized by the equivalent property:
\be
\bk \in \Lambda \quad \text{and} \quad k_i \geq 1 \quad \implies \quad \bk - \mathbf{e}_i \in \Lambda,
\ee
where \( \mathbf{e}_i \) denotes the standard unit vector in the \( i \)-th direction. These sets are particularly useful in various applications, as discussed, for instance, in \cites{ChkifaLower,Chkifa2016PolynomialAV}. Notably, any lower set must include the null multi-index \( \mathbf{0} \). For a lower set \( \Lambda \), we can write:
\[
\Lambda = \bigcup_{\bk \in \Lambda} \mathcal{B}_{\bk},
\]
where \( \mathcal{B}_{\bk} := \{\bh \in \mathbb{N}_0^d : \bh \leq \bk \} \) is the block-type set associated with \( \bk \).

\subsubsection*{Examples of Lower Sets}

Typical examples of lower sets include the isotropic sets 
$\left\{\bk\in\N_0^d: k_j \leq r ~\forall j\right\}$, 
$\{\bk\in\N_0^d: \sum_{j=1}^d k_j \leq r \}$, and
$\{\bk\in\N_0^d: \sum_{j=1}^d k_j^2 \leq r^2 \}$,
parameterized by $r \geq1$. These are the intersections of $\N_0^d$
with balls in $\R^d$ of radius $r$, for the $\ell_\infty$, $\ell_1$ 
and $\ell_2$ norms respectively. Another typical set is the hyperbolic 
cross $\{ \bk : \prod_{j=1}^d (1+k_j) \leq N\}$,
which is the union of all lower sets of cardinality less or equal than $N$.

\subsubsection*{Specific Types of Lower Sets}

We specifically refer lower sets of {\it block-type}, {\it cross-type}, and {\it simplex-type} 
which are defined as follow:
\be
\begin{array}{c}
\displaystyle \cB_\bk := \big\{\bh\in\N_0^d:\bh\leq\bk\big\},\quad 
\cC_\bk := \big\{h_j \kr_j :0\leq h_j \leq k_j\big\},\quad
\cS_{\bw,u} := \bigg\{\bh\in\N_0^d:\sum_{j=1}^d w_j h_j \leq u\bigg\}.
\end{array}
\ee
where:
\begin{itemize}
  \item \( \mathcal{B}_{\bk} \) and \( \mathcal{C}_{\bk} \) are parameterized by \( \bk \in \mathbb{N}_0^d \),
  \item \( \mathcal{S}_{\bw,u} \) is parameterized by positive weights \( w_1, \dots, w_d > 0 \) and an upper bound \( u > 0 \).
\end{itemize}

These sets are considered isotropic if \( \bk \) (for \( \mathcal{B}_{\bk} \) and \( \mathcal{C}_{\bk} \)) or \( \bw \) (for \( \mathcal{S}_{\bw,u} \)) is a scalar multiple of \( \mathbf{1} \). Figure~\ref{fig:typical_sets} illustrates these sets in dimension 2.

\subsubsection*{Maximal and Extremal Elements}

An element \( \bk \in \Lambda \) is called a \textit{maximal element} of the lower set \( \Lambda \) if it is maximal with respect to the partial order \( \leq \). Formally, \( \bk \) is maximal if:
\[
\bh \in \Lambda \quad \text{and} \quad \bk \leq \bh \quad \implies \quad \bh = \bk.
\]
Equivalently, \( \bk \) is maximal if \( \bk + \mathbf{e}_i \not\in \Lambda \) for all \( i = 1, \dots, d \), where \( \mathbf{e}_i \) denotes the standard unit vector in the \( i \)-th direction.

An element \( \bk \in \Lambda \) is called an \textit{extremal element} of \( \Lambda \) if it cannot be expressed as the midpoint of two distinct elements of \( \Lambda \). Formally, \( \bk \) is extremal if:
\be
2\bk = \bh + \bh' \quad \text{and} \quad \bh, \bh' \in \Lambda \quad \implies \quad \bh = \bh' = \bk.
\ee

The null multi-index \( \mathbf{0} \) is always an extremal element. However, if \( \bk + \mathbf{e}_j \in \Lambda \) for some \( j \in \supp(\bk) \),  then \( \bk \) is not extremal in \( \Lambda \). We denote by $\cE(\Lambda)$ (resp. $\cE_{\leq}(\Lambda)$) the set of extremal 
(resp. maximal) elements of $\Lambda$. Here, \(\mathrm{supp}(\boldsymbol{k})\) denote the support of \(\boldsymbol{k}\), defined as the subset of \(\{1, \dots, d\}\) where \(k_j \neq 0\).

A maximal element is not necessarily extremal, e.g. 
$(1,1)$ is maximal in $\{\bk\in\N_0^2: k_1+k_1\leq 2\}$
but not extremal since it is in the middle between 
 $(2,0)$ and $(0,2)$.
However, if \( \bk \) is maximal and \( \#(\supp(\bk)) = 1 \), then \( \bk \) is extremal. Moreover, any finite lower set has at least one maximal element. The set \( \Lambda \) is completely determined by its maximal elements:
\be
\label{unionBlocks}
\Lambda = \bigcup_{\bk \text{ maximal in } \Lambda} \mathcal{B}_{\bk}.
\ee

\section{Necessary Conditions for Exact Reconstruction}
\label{sec:Necessary conditions}

In this section, we derive lower bounds on the minimal number of rank-1 lattice points \( n^*_A(\Lambda) \), \( n^*_B(\Lambda) \), and \( n^*_C(\Lambda) \) required for exact reconstruction over \( \mathcal{F}^{\text{cos}}_\Lambda \) using Plan A, Plan B, and Plan C, respectively. These results offer necessary conditions for admissibility and provide insights into the sharpness of these bounds.

\subsection{Lower Bounds for Plans A, B, and C}

The requirement in \eqref{requirePlanA} ensures that the mapping:
\(
\bz \mapsto \bh \cdot \bz \pmod{n},
\)
is injective over \( \mathcal{M}(\Lambda) \). Similarly, the requirement in \eqref{requirePlanB} enforces injectivity over \( \Lambda \cup -\Lambda \) and the requirement in \eqref{requirePlanC} enforces injectivity over \( \Lambda \).

Thus, the following holds
\be
n \geq 
\begin{cases} 
\#\mathcal{M}(\Lambda) & \text{(Plan A admissibility)}, \\
\#(\Lambda \cup -\Lambda) & \text{(Plan B admissibility)}, \\
\#\Lambda& \text{(Plan C admissibility)}.
\end{cases}
\ee

If \( \bh \cdot \bz \equiv_n 0 \), then \( \bh \cdot \bz \equiv_n -\bh \cdot \bz \).
Assuming \( \mathbf{0} \notin \Lambda \) and \(\bz \mapsto \bh \cdot \bz \pmod{n}\) 
is injective over \( \mathcal{M}(\Lambda) \), the remainder 0 is excluded.
The same holds for $\Lambda \cup -\Lambda$. Consequently, the bounds 
are slightly improved to:
\(
n \geq \#\mathcal{M}(\Lambda) + 1_{\mathbf{0} \notin \Lambda} \quad \text{(Plan A admissibility)},
\)
\(
n \geq \#(\Lambda \cup -\Lambda) + 1_{\mathbf{0} \notin \Lambda} \quad \text{(Plan B admissibility)}.
\)

The lower bounds can be implied from reconstruction properties directly. For instance, 
\( (n, \bz) \) is admissible for Plan A if and only if the rank-1 lattice rule \( Q_n \) satisfies \eqref{condition_reconstruction} with \( (e_\bk)_{\bk \in \mathcal{M}(\Lambda)} \). This is equivalent to:
\(
\frac{1}{n} \Phi^\top \Phi = I_{\#\mathcal{M}(\Lambda)},
\)
where \( \Phi = \left(e_\bk(\mathbf{t}_i)\right)_{1 \leq i \leq n, \bk \in \mathcal{M}(\Lambda)} \). 
This implies \( \Phi \) has full column rank, hence
\(
n \geq \#\mathcal{M}(\Lambda).
\)

In general, for \( \Lambda \subset \mathbb{Z}^d \) and given functions \( (\alpha_\bk)_{\bk \in \Lambda} \), \( (\beta_\bk)_{\bk \in \Lambda} \) in \( \mathcal{F}_\Lambda \), and a cubature rule \( Q_n \) using \( m \) points \( \mathbf{t}_1, \ldots, \mathbf{t}_m \), we define the measurement matrices:
\(
\Phi_\alpha = \left(\alpha_\bk(\mathbf{t}_i)\right)_{1 \leq i \leq m, \bk \in \Lambda}, \quad
\Phi_\beta = \left(\beta_\bk(\mathbf{t}_i)\right)_{1 \leq i \leq m, \bk \in \Lambda}\).
The condition \eqref{bi-orthonormality} is equivalent to
\(
\Phi_\alpha^\top W \Phi_\beta = \text{diag}\left(c_\bk\right)_{\bk \in \Lambda},
\)
where \( W = \text{diag}\left(w_i\right)_{i=1, \ldots, m} \) are the cubature weights. 
The condition \( c_\bk \neq 0 \) for all \( \bk \in \Lambda \) implies that \( \Phi_\alpha \) 
and \( \Phi_\beta \) have full column rank, and thus \( m \geq \#\Lambda\). The 
tent-transformed/cosine-transformed lattice rule uses \( \lfloor n/2 + 1 \rfloor \) points, hence 
the admissibility of \( (n, \bz) \) for Plans B or C implies:
\(
\lfloor n/2 + 1 \rfloor \geq \#(\Lambda).
\)

We recapitulate that for any finite set $\Lambda\subset N_0^d$:
\be
\label{eq: n_min}
n^*_\square(\Lambda) \geq 
\begin{cases} 
l^*_A(\Lambda) := \#\mathcal{M}(\Lambda) + 1_{\mathbf{0} \notin \Lambda}, & \text{if } \square = A, \\
l^*_B(\Lambda) := 2\#\Lambda + \epsilon_{\Lambda,\mathbf{0}}, & \text{if } \square = B, \\
l^*_C(\Lambda) := 2\#\Lambda - 2, & \text{if } \square = C,
\end{cases}
\ee
where \( \epsilon_{\Lambda, \mathbf{0}} = 1 \) if \( \mathbf{0} \notin \Lambda \), and \( \epsilon_{\Lambda, \mathbf{0}} = -1 \) otherwise.

\subsection{Sharpness of the necessary condition}
For $\bk, \bk' \in \mathbb{Z}^d$, $\bk' \le \bk$ is to be interpreted component-wise, i.e., $k'_j \le k_j$ for all $j$.
We introduce the notation $\cB_{\bk}$, $\cC_{\bk}$ and $\cS_{\bone,k}$
for $\cB_{\bk} := \{\bh\in\N_0^2: \bh \leq \bk\}$,
$\cC_{\bk} := \{\bh\in\cB_{\bk}: h_1h_2=0\}$
where $\bk\in\N_0^2$ and  
$\cS_{\bone,k}=\{\bh\in\N_0^2: h_1 + h_2 \leq k\}$
where $k\in\N_0$. Mirrored sets are
$\cM(\cB_{\bk}) = \{\bh\in\Z^2:  -\bk  \leq \bh \leq \bk \}$,
$\cM(\cC_{\bk}) = \{\bh\in \cM(\cB_{\bk}):  h_1h_2=0 \}$,
and $\cM(\cS_{\bone,k})=\{\bh\in\Z^2: |h_1| + |h_2| \le k\}$. 
Cardinalities of all these sets are easily verified 
\be
\begin{array}{ll}
\#(\cB_{\bk})=(k_1+1)(k_2+1),&\quad \#(\cM(\cB_{\bk}))=(2k_1+1)(2k_2+1),\\
\#(\cC_{\bk})=k_1+k_2+1,&\quad \#(\cM(\cC_{\bk}))=2(k_1+k_2)+1, \\
\#(\cS_{\bone,k})={(k+1)(k+2)}/2,& \quad   \#(\cM(\cS_{\bone,k})) = 2k^2+2k+1.
\end{array}
\ee

The lower bound $n^*_A(\Lambda) \geq l^*_A(\Lambda)$ is sharp. Indeed,
\begin {itemize}
\item for any $\cB_{\bk}$, $(n,\bz)$ with $n=\#(\cM(\cB_{\bk}))$ and 
$\bz=(1,2k_1+1)$ is admissible for plan A.

\item for any $\cS_{\bone,k}$, $(n,\bz)$ with $n=\#(\cM(\cS_{\bone,k}))$ 
and $\bz=(1,2k+1)$ \footnote{or $\bz=(k,k+1)$ which yields the Padua points in $\R^2$, 
see \cite{Padua}.}
 is admissible for plan A.
\end {itemize}
These results can be easily checked using elementary arguments.

In general, the optimal value $n^*_A(\Lambda)$ can however by arbitrary large 
compared to $l^*_A(\Lambda)$. For instance, 
\begin {itemize}
\item for any $\cC_\bk$, $(n,\bz)$ with $n=\#(\cB_{\bk})+1$ and 
$\bz=(1,k_1+1)$ is admissible for plan A. However, for any $n \leq \#(\cB_\bk)$ and 
any $\bz$, the admissibility cannot hold. In particular, we have 
$n^*_A(\cC_\bk) - l^*_A(\cC_\bk) = k_1 k_2 - (k_1+k_2)$.
\end{itemize}
The admissibility is easily checked. The justification for the 
second claim is given in the Appendix~\ref{prop: c_k}.




\begin{figure}
\centering
 \begin{subfigure}{0.3\textwidth}
     \includegraphics[width=\textwidth]{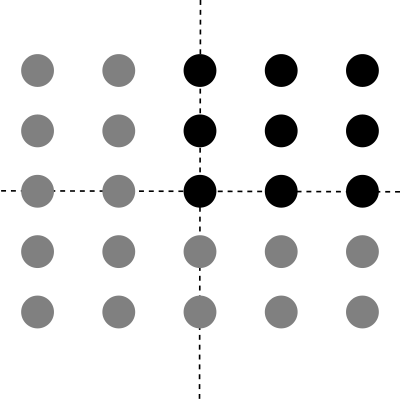}
     \caption{block}
     \label{fig:block}
 \end{subfigure}
 \hfill
 \begin{subfigure}{0.3\textwidth}
     \includegraphics[width=\textwidth]{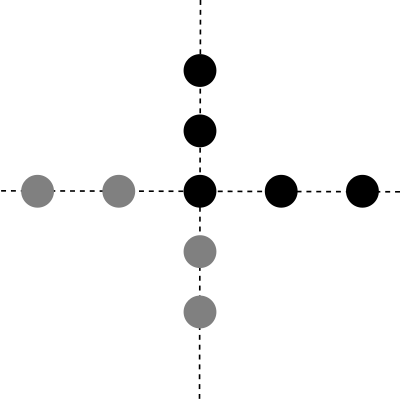}
     \caption{cross}
     \label{fig:cross}
 \end{subfigure}
 \hfill
 \begin{subfigure}{0.3\textwidth}
     \includegraphics[width=\textwidth]{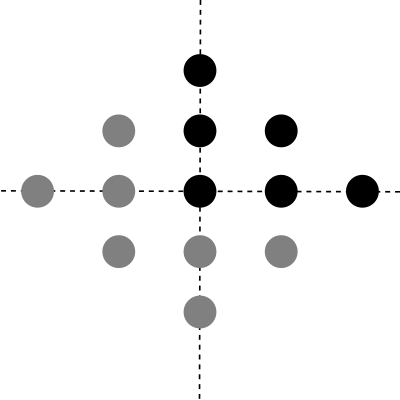}
     \caption{simplex}
     \label{fig:simplex}
 \end{subfigure}
\caption{ Sets $\cB_{\bk}$, $\cC_{\bk}$ and $\cS_{\bone,k}$
and associated mirrored sets in dimension 2.}
\label{fig:typical_sets}
\end{figure}

\section{Equivalence between Admissibility Requirements for lower sets}
\label{sec:equivalence_plans_ABC}

In this section, we establish equivalence properties between the admissibility requirements \eqref{requirePlan0}, \eqref{requirePlanA}, \eqref{requirePlanB}, and \eqref{requirePlanC} when \( \Lambda \) is a lower set, with a particular focus on simplex-type sets.

\subsection{Equivalences for General Lower Sets}

We begin with the following result:

\begin{proposition}
\label{LemmaPlanA0Lower}
Let \( \Lambda \subset \mathbb{N}_0^d \) be a lower set, \( n \in \mathbb{N} \), and \( \bz \in \mathbb{Z}^d \). The following statements are equivalent:
\begin{itemize}
    \item \( (n, \bz) \) is admissible for Plan 0 for \( \Lambda \oplus \Lambda \),
    \item \( (n, \bz) \) is admissible for Plan A for \( \Lambda \).
\end{itemize}
\end{proposition}

This equivalence follows from Proposition~\ref{propSum}, which states that:
\[
\mathcal{M}(\Lambda \oplus \Lambda) = \mathcal{M}(\Lambda) \oplus \mathcal{M}(\Lambda) = \mathcal{M}(\Lambda) \ominus \mathcal{M}(\Lambda),
\] where 
\(
\Lambda \ominus \Lambda' \;=\; \{\bk - \bk' \,:\, \bk\in \Lambda,\,\bk'\in \Lambda'\},
\) is the \textit{difference set}.
In particular, the condition \( \bh \cdot \bz \not\equiv_n 0 \) for all \( \bh \in \mathcal{M}(\Lambda \oplus \Lambda) \setminus \{\mathbf{0}\} \) is equivalent to \( \bh \cdot \bz \not\equiv_n \bh' \cdot \bz \) for all \( \bh, \bh' \in \mathcal{M}(\Lambda) \) with \( \bh \neq \bh' \).

Let us discuss a few specific cases to illustrate these equivalences:
\begin{itemize}
\item $d=2$ and $\Lambda =  \mathcal{C}_{\bk}$: we have that 
$\cB_\bk  \subset \cC_\bk \oplus \cC_\bk $, hence
$\cB_\bk \ominus \cB_\bk = \cM (\cB_\bk) \subset \cM(\cC_\bk \oplus \cC_\bk)$. 
As a result, if $(n,\bz)$ is admissible for plan A for $\cC_\bk$ 
then 
$\bh \cdot \bz \not\equiv_n \bh' \cdot \bz$ for $\bh \neq \bh' \in \cB_\bk$. 
This implies that $n\geq \#(\cB_\bk)$. 
We actually have
\be
n_A^*( \mathcal{C}_{\bk}) = (k_1+1)(k_2+1) +1.
\ee
The reason why $n_A^*( \mathcal{C}_{\bk}) \neq \#(\cB_\bk)$ is 
explained in the Appendix~\ref{prop: c_k}.

\item \( d \geq 2 \), \( \Lambda = \mathcal{B}_{\bk} \):
we have that \( \mathcal{B}_{2\bk} \ominus \mathcal{B}_{2\bk} = \mathcal{M}(\mathcal{B}_{2\bk}) = \mathcal{M}(\mathcal{B}_{\bk} \oplus \mathcal{B}_{\bk}) \). Thus, \( (n, \bz) \) is admissible for Plan A for \( \mathcal{B}_{\bk} \) if and only if
   \(
   \bh \cdot \bz \not\equiv_n \bh' \cdot \bz\) for all 
   \(\bh \neq \bh' \in \mathcal{B}_{2\bk}\).
 We recover again \( n \geq \#(\mathcal{B}_{2\bk}) = \#(\mathcal{M}(\mathcal{B}_{\bk}))  \). The 
 pair \( (n, \bz) \) with \( n = \#(\mathcal{M}(\mathcal{B}_{\bk})) \) and \( \bz = (z_1, \dots, z_d) \), where 
 \( z_1 = 1 \) and \( z_i = \prod_{j=1}^{i-1} (2k_j + 1) \) for \( i = 2, \dots, d \), yields Plan A 
 admissibility. Indeed, $\bh \mapsto \bh\cdot\bz$ is one-to-one from $\mathcal{B}_{2\bk}$ onto
$\{0,\dots,\#(\mathcal{B}_{2\bk})-1\}$. We therefore have
\be
n_A^*(\mathcal{B}_{\bk}) = (2k_1 + 1) \cdots (2k_d + 1).
\ee

\item \( d \geq 2 \), \( \Lambda = \mathcal{S}_{{\bw},u} \): 
we have that
   \(
   \mathcal{S}_{\bw, 2u - \max_i(w_i)} \subset \mathcal{S}_{\bw,u} \oplus \mathcal{S}_{\bw,u} \subset \mathcal{S}_{\bw, 2u}.
   \)
   Therefore, if \( (n, \bz) \) is admissible for Plan 0 for \( \mathcal{S}_{\bw, 2u} \), then \( (n, \bz) \) is admissible for Plan A for \( \mathcal{S}_{\bw,u} \). This implies:
   \be
   \#(\mathcal{M}(\mathcal{S}_{\bw,u})) \leq n_A^*(\mathcal{S}_{\bw,u}) \leq n_0^*(\mathcal{S}_{\bw, 2u}) \leq p_0^*(\mathcal{S}_{\bw, 2u}),
   \ee
   where \( p_0^*(\mathcal{S}_{\bw, 2u}) \) is the smallest prime larger than
\( \frac{\#(\mathcal{M}(\mathcal{S}_{\bw, 2u})) + 1}{2} \).
This follows directly from \eqref{eq: upper bounds}
and the fact \( \mathcal{S}_{\bw, 2u}\) is lower. 

\end{itemize}
This next result is also straightforward. 

\begin{proposition}
Let \( \Lambda \subset \mathbb{N}_0^d \) be a lower set, \( n \in \mathbb{N} \), and \( \bz \in \mathbb{Z}^d \). Then:
\be
\bh \cdot \bz \not\equiv_n \bh' \cdot \bz \quad \forall \bh \neq \bh' \in \Lambda
\quad \implies \quad
\bh \cdot \bz \not\equiv_n 0 \quad \forall \bh \in \mathcal{M}(\Lambda) \setminus \{\mathbf{0}\}.
\ee
An equivalence holds if \( \Lambda = \mathcal{B}_{\bk} \) for some \( \bk \in \mathbb{N}_0^d \).
\end{proposition}

\begin{proof}
For any \( \bh \in \mathcal{M}(\Lambda) \setminus \{\mathbf{0}\} \), we can decompose it as \( \bh = \bh^+ - \bh^- \), where \( \bh^+ = \max(\bh, \mathbf{0}) \) and \( \bh^- = \max(-\bh, \mathbf{0}) \). Since \( \bh^+, \bh^- \leq |\bh| \), and \( |\bh| \in \Lambda \), the lower property ensures \( \bh^+, \bh^- \in \Lambda \). By construction, \( \bh^+ \neq \bh^- \), and
\(
\bh^- \cdot \bz \not\equiv_n \bh^+ \cdot \bz \quad \implies \quad \bh \cdot \bz \not\equiv_n 0.
\)
If \( \Lambda = \mathcal{B}_{\bk} \), note that \( \bh \neq \bh' \in \Lambda \) is equivalent to \( \mathbf{0} \neq (\bh - \bh') \in \mathcal{M}(\Lambda) \).
\end{proof}

\subsubsection*{Equivalence Between Plans B and C}

The following results address the equivalence between Plans B and C, i.e., requirements \eqref{requirePlanB} and \eqref{requirePlanC}. The techniques rely on those used earlier but are more intricate due to additional constraints.

\begin{proposition}
\label{LemmaPlanBLower}
Let \( \Lambda \subset \mathbb{N}_0^d \) be a lower set, \( n \in \mathbb{N} \), and \( \bz \in \mathbb{Z}^d \). The following statements are equivalent:
\begin{itemize}
    \item \( (n, \bz) \) is admissible for Plan B,
    \item \( (n, \bz) \) is admissible for Plan C and \( 2 \bh \cdot \bz \not\equiv_n 0 \) for all \( \bh \in \mathcal{E}(\Lambda) \setminus \{\mathbf{0}\} \).
\end{itemize}
Moreover, if \( n \) is odd, the additional condition \( 2 \bh \cdot \bz \not\equiv_n 0 \) is not required.
\end{proposition}

\begin{proof}
It is straightforward that Plan B admissibility implies Plan C admissibility and that \( \bh \cdot \bz \not\equiv_n -\bh \cdot \bz \) for all \( \bh \in \Lambda \setminus \{\mathbf{0}\} \), as noted in Table~\ref{tab:IncrementalABC}.

For the converse, assume \( (n, \bz) \) is admissible for Plan C and satisfies the additional condition, but is not admissible for Plan B. By Table~\ref{tab:IncrementalABC}, there exists \( \bh \in \Lambda \setminus \{\mathbf{0}\} \) and \( \sigma \in \mathcal{S}_{\bh}^* \) such that
\(
\bh \cdot \bz \equiv_n \sigma(\bh) \cdot \bz.
\)
We have \( \sigma(\bh) = \bh^+ - \bh^- \) and \( \bh = \bh^+ + \bh^- \), where
\(
\bh^+ = \max(\sigma(\bh), \mathbf{0}) \), 
\( \bh^- = \max(-\sigma(\bh), \mathbf{0})\), both in \( \mathbb{N}_0^d \). Note that \( \bh^- \neq \mathbf{0} \), as \( \bh^- = \mathbf{0} \) implies \( \sigma = \text{id} \), which is excluded.

The condition \( \bh \cdot \bz \equiv_n \sigma(\bh) \cdot \bz \) reformulates as
\(
2 \bh^- \cdot \bz \equiv_n 0 \equiv_n \bzero\cdot \bz .
\)
If \( 2 \bh^- \in \Lambda \), this contradicts Plan C admissibility. In general, since \( \Lambda \) 
is lower, then \( \bh^- \in \Lambda \). If \( \bh^- \not\in \mathcal{E}(\Lambda) \), then \( 2 \bh^- = \bk + \bk' \) for some \( \bk \neq \bk' \in \Lambda \). This implies
\(
\bk \cdot \bz \equiv_n -\bk' \cdot \bz,
\)
contradicting Plan C admissibility. Thus, the equivalence is established.

If \( n \) is odd, \( 2 \bh^- \cdot \bz \not\equiv_n 0 \) is equivalent to \( \bh^- \cdot \bz \not\equiv_n 0 \), which is already included in Plan C admissibility. The proof is complete.
\end{proof}

\subsubsection*{Additional Considerations}

In general, the requirement \( 2 \bh \cdot \bz \not\equiv_n 0 \) for all \( \bh \in \mathcal{E}(\Lambda) \setminus \{\mathbf{0}\} \) can be challenging, as it requires identifying the extremal elements of \( \Lambda \). Therefore, it is often advantageous to consider \( n \) odd, where this condition is not needed.

Since \( n_C^*(\Lambda) \leq n_B^*(\Lambda) \) and \( n_C^*(\Lambda, \bz) \leq n_B^*(\Lambda, \bz) \) for any \( \Lambda \subset \mathbb{N}_0^d \) and any \( \bz \), the following holds:
\be
n_C^*(\Lambda, \bz) \text{ odd } \implies n_B^*(\Lambda, \bz) = n_C^*(\Lambda, \bz) \quad \text{and} \quad 
n_C^*(\Lambda) \text{ odd } \implies n_B^*(\Lambda) = n_C^*(\Lambda).
\ee

\begin{remark}
In general, \( n_C^*(\Lambda) = n_B^*(\Lambda) \) does not always hold. For example, consider \( \Lambda = \{0, 1\} \), a lower set in dimension 1. For Plan C, we require \( z \not\equiv_n 0 \), so \( (n=2, \bz=1) \) is admissible, and \( n_C^*(\Lambda) = 2 \). For Plan B, we additionally require \( z \not\equiv_n -z \), implying \( (n=3, \bz=1) \) is admissible, and \( n_B^*(\Lambda) = 3 \).
\end{remark}

\begin{remark}
Given $\Lambda \subset \N_0^d $ lower and $n\in\N$ odd, having 
$\bh'\cdot \bz \equiv_n \bh\cdot\bz$ for some $\bh' \neq \bh \in \cM(\Lambda)$ 
in the same $\cS_\bk$ is equivalent to $\widetilde \bh \cdot \bz \equiv_n 0$ 
where $\widetilde \bh\neq\bzero$ is defined by $\widetilde h_j=k_j$ if $h_j' = -h_j$ 
and $\widetilde h_j = 0$ otherwise. Since $\widetilde \bh\in\Lambda$, this is ruled out if 
plan C admissibility holds. 
%
\end{remark}
The implications when $n\in\N$ is odd are summarized in Table~\ref{tab:Specialized_Lower}.
\begin{table}[h!]
\centering
\begin{tabular}{| l l |c|c|c|}
\hline
Requirement & & Plan C & Plan B & Plan A \\ \hline
\( \bh' \cdot \bz \not \equiv_n \sigma(\bh) \cdot \bz \) & \( \bh' \neq \bh \in \Lambda,~ \sigma \in \mathcal{S}_{\bh} \) & \checkmark & \checkmark & \checkmark \\
\( \bh' \cdot \bz \not \equiv_n \bh \cdot \bz \) & \( \bh', \bh \in \mathcal{M}(\Lambda),~{\rm s.t.}~ |\bh'| \neq |\bh| \) & & & \checkmark \\ \hline
\end{tabular}
\caption{Specialization of Table~\ref{tab:IncrementalABC} to \( \Lambda \) lower and \( n \) odd.}
\label{tab:Specialized_Lower}
\end{table}

\subsection{Equivalence for Isotropic Simplex-Type Sets}
We now specialize to simplex-type sets. For these sets, minimal conditions are required to establish equivalence between plans. Specifically, for \( \mathcal{S}_{\mathbf{1},k} \), combining Proposition~\ref{LemmaPlanBLower} with the fact that 
\(
\mathcal{E}(\mathcal{S}_{\mathbf{1},k}) \setminus \{\mathbf{0}\} = \{k\mathbf{e}_j : j=1,\dots,d\},
\)
(see \eqref{extremal_S1k}), we obtain the equivalence between the following statements:
\begin{itemize}
    \item \( (n,\bz) \) is admissible for Plan B, 
    \item \( (n,\bz) \) is admissible for Plan C and \( 2k z_j \not\equiv_n 0 \) for all \( j=1,\dots,d \).
\end{itemize}
The last condition is not needed if \( n \) is odd. A stronger result holds, as stated below.

\begin{proposition}
\label{LemmaEquivSimplex}
Let \( k, n \in \mathbb{N} \), and \( \bz \in \mathbb{Z}^d \). For \( \Lambda = \mathcal{S}_{\mathbf{1},k} \), the following statements are equivalent:
\begin{itemize}
    \item \( (n,\bz) \) is admissible for Plan A
    \footnote{also equivalent to $(n,\bz)$ is admissible for plan 0 for $S_{\bone,2k}$ 
    ($=S_{\bone,k}\oplus S_{\bone,k})$.},
    \item \( (n,\bz) \) is admissible for Plan B,
    \item \( (n,\bz) \) is admissible for Plan C and \( 2k z_j \not\equiv_n 0 \) for all \( j=1,\dots,d \).
\end{itemize}
Moreover, if \( n \) is odd, the last condition is not needed.
\end{proposition}

\begin{proof}
The forward implications are immediate. We focus on proving the converse implications. Assume \( (n,\bz) \) is admissible for Plan C and satisfies the required condition but is not admissible for Plan A. Let \( \bh' \neq \bh \in \mathcal{M}(\mathcal{S}_{\mathbf{1},k}) \) such that
\(
\bh' \cdot \bz \equiv_n \bh \cdot \bz.
\)
Decompose \( \bh = \bh^+ - \bh^- \) and \( \bh' = \bh'^+ - \bh'^- \), leading to
\(
\bk \cdot \bz \equiv_n \bk' \cdot \bz,
\)
where \( \bk = \bh^+ + \bh'^- \) and \( \bk' = \bh'^+ + \bh^- \). Both \( \bk, \bk' \in \mathbb{N}_0^d \), and \( \bk \neq \bk' \). By subtracting \( \min(\bk, \bk') \) from \( \bk \) and \( \bk' \), and swapping if necessary, we can assume \( \supp(\bk) \cap \supp(\bk') = \emptyset \) and \( \|\bk\|_{\ell_1} \leq \|\bk'\|_{\ell_1} \). Then
\(
2 \|\bk\|_{\ell_1} \leq \|\bk\|_{\ell_1} + \|\bk'\|_{\ell_1} = \|\bh\|_{\ell_1} + \|\bh'\|_{\ell_1} \leq 2k,
\)
hence \( \bk \in \mathcal{S}_{\mathbf{1},k} \). If \( \bk' \in \mathcal{S}_{\mathbf{1},k} \), this contradicts Plan C admissibility. 

Consider the case \( \bk' \not\in \mathcal{S}_{\mathbf{1},k} \), meaning \( \|\bk'\|_{\ell_1} > k \). Without loss of generality, assume \( \supp(\bk') = \{1,\dots,J\} \). We apply a shifting procedure: set \( {\by} = \bk \) and \( {\by}' = \bk' \), then for each \( j = 1, \dots, J \), while \( y'_j \geq 1 \) and \( \|{\by}'\|_{\ell_1} > k \), perform
\(
{\by} \gets {\by} - \mathbf{e}_j, ~ {\by}' \gets {\by}' - \mathbf{e}_j.
\)
This produces \( {\by}, {\by}' \in \mathbb{N}_0^d \) such that:
\[
\|{\by}\|_{\ell_1} + \|{\by}'\|_{\ell_1} = \|\bk\|_{\ell_1} + \|\bk'\|_{\ell_1}, \quad \|{\by}'\|_{\ell_1} = k.
\]
Thus, \( {\by} \in \mathcal{M}(\mathcal{S}_{\mathbf{1},k}) \) and \( {\by}' \in \mathcal{S}_{\mathbf{1},k} \). Since \( {\by} \cdot \bz \equiv_n {\by}' \cdot \bz \), then if $|\by| \neq \by'$
we reach a contradiction with plan C admissibility. 
We observe that $|\by| \neq \by'$ will always hold 
unless $J=1$, $\bk = \bzero$ and $\bk' = 2k {\kr}_1$.
In this case
\(
\bk' \cdot \bz \equiv_n \bk \cdot \bz,
\)
is equivalent to \( 2k z_1 \equiv_n 0 \), which contradicts the assumption on \( k \) and \( \bz \).

When $n$ is odd,  the assumption 
$2 k z_1\not\equiv_n 0,\dots,2 k z_d \not\equiv_n  0$ reduces to 
$ k z_1\not\equiv_n$,\dots, $k z_d \not\equiv_n  0$ 
which is incorporated in plan C admissibility. The proof is complete.
\end{proof}

The arguments presented in the proof simplify the requirements for Plan A. In the first case of the proof, i.e., \(\|\bk\|_{\ell_1} \leq \|\bk'\|_{\ell_1} \leq k\), we can define 
\(\by = \bk + \delta \mathbf{e}_1\) and \(\by' = \bk' + \delta \mathbf{e}_1\), where \(\delta = (k - \|\bk'\|_{\ell_1})\), ensuring \(\|\by\|_{\ell_1} \leq \|\by'\|_{\ell_1} = k\). Similarly, this holds for \(\by\) and \(\by'\) produced by the procedure in the second case of the proof. In both cases, \(\supp(\by) \cap \supp(\by')\) contains at most one index.

If instead, we continue updating \(\by\) and \(\by'\) while \(\|\by'\|_{\ell_1} > \|\by\|_{\ell_1} + 1\), rather than halting when \(\|\by'\|_{\ell_1} = k\) (and apply this even if \(\|\bk\|_{\ell_1} \leq \|\bk'\|_{\ell_1} \leq k\)), we obtain \(\|\by'\|_{\ell_1} - \|\by\|_{\ell_1} \in \{0, 1\}\), while preserving the property that \(\supp(\by) \cap \supp(\by')\) contains at most one index.

Thus, for \(\Lambda = \cS_{\bone, k}\), \((n, \bz)\) is admissible for Plan A if and only if:
\be
\label{simplex_planA_1}
\bh' \cdot \bz \not\equiv_n \sigma(\bh) \cdot \bz, \quad \text{for } \quad \bh' \neq \bh \in \Lambda,~\sigma \in \mathcal{S}_{\bh},~ {\rm s.t.}~\|\bh'\|_{\ell_1} = k,
\ee
or equivalently:
\be
\label{simplex_planA_2}
\bh' \cdot \bz \not\equiv_n \sigma(\bh) \cdot \bz, \quad \text{for } \quad \bh' \neq \bh \in \Lambda,~\sigma \in \mathcal{S}_{\bh},~ {\rm s.t.}~ \|\bh'\|_{\ell_1} - \|\bh\|_{\ell_1} \in \{0,1\}.
\ee

Additionally, we may require that \(\supp(\bh) \cap \supp(\bh')\) consists of at most one index.

We introduce functions $\alpha_\bk$ and $\beta_\bk$ by
$\alpha_\bk (\bx) = \sqrt{2}^{|\bk|_0} \prod_{j=1}^d \cos (2 \pi k_j x_j)$,
$\beta_\bk (\bx) = \sqrt{2}^{|\bk|_0} \cos (2 \pi \bk \cdot {\bx})$.
Given a rank-1 lattice rule \( Q_n^* \) associated with $n$ odd and $\bz\in \Z^d$, 
we define the measurement matrices:
\(\Phi_\alpha = \left(\alpha_\bk(\mathbf{t}_i)\right)_{0 \leq i \leq n-1, \bk \in \cS_{\bone,k}}\),
\(\Phi_\beta = \left(\beta_\bk(\mathbf{t}_i)\right)_{0 \leq i \leq n-1, \bk \in \cS_{\bone,k}}\)
and
\(\Phi_{\alpha,2} = \left(\alpha_\bk(\mathbf{t}_i)\right)_{0 \leq i \leq n-1, \bk \in \cS_{\bone,2k}}\).
Lemma \ref{LemmaEquivSimplex} basically states that 
\be
\frac1n \Phi_{\alpha,2}^\top \bone = e_1 
\quad\iff \quad
\frac1n \Phi_{\alpha}^\top \Phi_{\alpha}  = I_{s_k} 
\quad\iff \quad
\frac1n \Phi_{\alpha}^\top \Phi_{\beta}  = I_{s_k},
\ee
where $s_k = \#(\cS_{\bone,k})$ and $I_{s_k}$ is the $s_k\times s_k$ identity matrix.
We group the indices $\bk$ in the definition of $\Phi_\alpha$ by increasing 
${\ell_1}$-norm and accordingly view the symmetric matrix $\Phi_{\alpha}^\top \Phi_{\alpha}$ as a 
$(k+1)\times (k+1)$ block matrix. Characterization \eqref{simplex_planA_1}
basically states that for the symmetric matrix $\frac1n \Phi_{\alpha}^\top \Phi_{\alpha}$ to be 
equal to the identity 
it suffices to have its far right $(k+1)$ block equal to $\bzero,\dots,\bzero, I_{m}$ with 
$m = \#\{\bh\in\cS_{\bone,k}:\|\bh\|_{\ell_1}=k \}$ the number of maximal elements of 
$\cS_{\bone,k}$. Characterization \eqref{simplex_planA_2}
states that it also suffices to have the diagonal blocks equal to identity matrices 
while the immediate upper diagonal blocks equal to zeros matrices. 


\subsection{Equivalence for Anisotropic Simplex-Type Sets}

The techniques used in the proof of Proposition~\ref{LemmaEquivSimplex} can be extended to general simplex-type sets. For \( w_1, \dots, w_d > 0 \), define the weighted \( \ell_1 \)-norm:
\(
\|{\bx}\|_{\bw, \ell_1} := \sum_{j=1}^d w_j |x_j|,
\)
which is a valid norm over \( \mathbb{R}^d \). Using this, the anisotropic simplex-type set 
and associated mirrored set are given by:
\[
\mathcal{S}_{\bw, u} = \{\bk \in \mathbb{N}_0^d : \|\bk\|_{\bw, \ell_1} \leq u\}, \quad \mathcal{M}(\mathcal{S}_{\bw, u}) = \{\bh \in \mathbb{Z}^d : \|\bh\|_{\bw, \ell_1} \leq u\}.
\]
The proof remains unchanged when replacing \( \|\cdot\|_{\ell_1} \) with \( \|\cdot\|_{\bw, \ell_1} \) and \( k \) with \( u \), except when considering \( \bk' \notin \mathcal{S}_{\bw, u} \). In this case, two scenarios arise:
\begin{itemize}
    \item \( \|\bk\|_{\bw, \ell_1} + \|\bk'\|_{\bw, \ell_1} \leq 2u - \alpha \),
    \item \( \|\bk\|_{\bw, \ell_1} + \|\bk'\|_{\bw, \ell_1} > 2u - \alpha \),
\end{itemize}
where \( \alpha = \max_{i \in \supp(\bk')} w_i \). In the first case, the shifting procedure used 
in the proof of Proposition~\ref{LemmaEquivSimplex} yields \( {\by} \) and \( {\by}' \) 
satisfying \( \mathbf{0} \leq {\by}' \leq \bk' \) and
\[
u - \alpha < \|{\by}'\|_{\bw, \ell_1} \leq u, \quad \|{\by}\|_{\bw, \ell_1} + \|{\by}'\|_{\bw, \ell_1} = \|\bk\|_{\bw, \ell_1} + \|\bk'\|_{\bw, \ell_1} \leq 2u - \alpha,
\]
and thus \( \|{\by}\|_{\bw, \ell_1} \leq u \). Also, \( {\by} \neq {\by}' \) unless \( J = 1 \), \( \bk = \mathbf{0} \), and \( \bk' = 2 \lfloor u/w_1 \rfloor \mathbf{e}_1 \) (assuming \( w_1 \leq u \)).
The second case originates from \( \bh, \bh' \) such that 
\(
\|\bh\|_{\bw, \ell_1} + \|\bh'\|_{\bw, \ell_1} > 2u - \alpha
\). We note that $\alpha \leq \alpha_{\bh, \bh'}$ where
\(
\alpha_{\bh, \bh'} := \max_{i \in\supp(\bh)\cup \supp(\bh') } (w_i)\),
with possible equality. We therefore have the following result.

\begin{proposition}
\label{LemmaEquivSimplexGeneral}
Let \( w_1, \dots, w_d,u > 0 \), \( n \in \mathbb{N} \), and \( \bz \in \mathbb{Z}^d \). For \( \Lambda = \mathcal{S}_{\bw, u} \), the following are equivalent:
\begin{itemize}
    \item \( (n, \bz) \) is admissible for Plan A,
    \item \( (n, \bz) \) is admissible for Plan C, \( 2 \lfloor u/w_j \rfloor z_j \not\equiv_n 0 \) for \( j \) s.t. \( u \geq w_j \), and
    \[
    \bh' \cdot \bz \not\equiv_n \bh \cdot \bz, \quad \text{for } \bh' \neq \bh \in \mathcal{M}(\Lambda) ~~\text{ with }~~ 
    \|\bh\|_{\bw, \ell_1} + \|\bh'\|_{\bw, \ell_1} > 2u - \alpha_{\bh, \bh'}.
    \]
\end{itemize}
If \( n \) is assumed odd, the second condition is incorporated in plan C admissibility.
\end{proposition}

Characterizations such as \eqref{simplex_planA_1} and
\eqref{simplex_planA_2} can also be derived for $\Lambda = \mathcal{S}_{\bw, u} $.
Let \( \mathcal{M}(\mathcal{E}_{\leq}(\Lambda)) \) be the mirrored set of extremal 
elements of $\Lambda = \mathcal{S}_{\bw, u}$. Then, $(n,\bz)$ is admissible for plan A iff
\be
\displaystyle \bh' \cdot\bz \not \equiv_n \bh \cdot\bz,
\quad \mbox{ for } \quad
\bh \in \cM(\Lambda), \;\;
\bh' \in \cM(\cE_{\leq} ( \Lambda)),
 \;\;  \bh\neq \bh',
\label{requirePlanASimplexType}
\ee
or equivalently
\be
\bh' \cdot\bz \not\equiv_n \bh\cdot\bz \quad \mbox{for} \quad
\bh' \neq \bh \in \cM(\Lambda) 
~~ {\rm s.t.}~~
\left|
\|\bh'\|_{\bw,\ell^1} - \|\bh\|_{\bw,\ell^1}
\right| \leq \alpha_{\bh,\bh'}.
\label{requirePlanASimplexType2}
\ee

We will simply show that if 
$\bh'\cdot \bz \equiv_n \bh\cdot \bz$ for some 
$\bh' \neq \bh \in \cM(\cS_{\bw,u})$, the same holds with 
$\bh$ and $\bh'$ as in \eqref{requirePlanASimplexType} or in 
\eqref{requirePlanASimplexType2}.
The first is easily checked. 
Indeed, up to swap and 
multiply by $-1$ if necessary, we assume 
$\|\bh \|_{\bw,\ell_1} \leq \|\bh' \|_{\bw,\ell_1}$ and $h_i'\geq0$ where 
$i$ is s.t. $w_i$ is minimal among the $w_j$. Define
$\by = \bh+ \delta \mathbf{e}_i$ and $\by' = \bh'+ \delta \mathbf{e}_i$ where $\delta$
is the largest integer s.t. $|\by'|\in \cS_{\bw,u}$. This ensures 
$|\by'|$ is maximal, $\|\by\|_{\bw,\ell_1} \leq \|\by'\|_{\bw,\ell_1} \leq u $
and $\by'\cdot \bz \equiv_n \by\cdot \bz$. 

As for the second, we construct $\bk$ and $\bk'$ as in 
the proof of Proposition \ref{LemmaEquivSimplex}, and 
analyzes the two cases $(\|\bk \|_{\bw,\ell_1} +\| \bk'\|_{\bw,\ell_1})$
is $\leq$ or $>$ than $2u- \alpha$ with \( \alpha = \max_{i \in \supp(\bk')}(w_i)\).
In the first case, we apply the shifting procedure by updating 
$\by$ and $\by'$ while $\|\by' \|_{\bw,\ell_1} - \|\by \|_{\bw,\ell_1} > \alpha$, 
instead of halting when $\|\by' \|_{\bw,\ell_1} \leq u$. This leads to 
$| \|\by' \|_{\bw,\ell_1} - \|\by \|_{\bw,\ell_1}|\leq \alpha$, and we verify that  
$\|\by \|_{\bw,\ell_1} \leq u$ and $\|\by' \|_{\bw,\ell_1} \leq u$. The other case 
$\|\bk \|_{\bw,\ell_1} +\| \bk'\|_{\bw,\ell_1} > 2u- \alpha$, which stems from
$\|\bh \|_{\bw,\ell_1} +\| \bh'\|_{\bw,\ell_1} > 2u- \alpha$, implies that
$\bh$ and $\bh'$ are necessarily s.t. $u-\alpha < \|\bh \|_{\bw,\ell_1}, \|\bh' \|_{\bw,\ell_1} \leq u$, 
thus implying $\left|\|\bh'\|_{\bw,\ell^1} - \|\bh\|_{\bw,\ell^1}\right| < \alpha$.
We conclude in view of $\alpha \leq \alpha_{\bh, \bh'}$.

\section{A Specialized CBC Construction for Lower Sets}
\label{sec:specialized_cbc}

In this section, we present a \emph{specialized component-by-component (CBC)} algorithm for constructing rank-1 lattices, targeting lower index sets \(\Lambda \subset \mathbb{N}_0^d\) for Chebyshev expansions. The algorithm identifies a generator \(\bz\) and modulus \(n\) that satisfy admissibility criteria for \(\Lambda\) under plans \(0, A, B, C\). While it does not guarantee the minimal \(n\), it provides a practical, polynomial-time solution for high-dimensional problems and large index sets.

The approach exploits the hierarchical structure of lower sets, enabling incremental admissibility checks across dimensions. By maintaining dictionaries to track intermediate results, the algorithm reduces redundant aliasing checks and improves computational efficiency.

Our approach shares conceptual similarities with certain CBC strategies presented in the literature; for instance, Algorithm~3.8 in \cite{kaemmerer2015} or the one proposed in \cite{Cools2019FastCC} also iterate dimension by dimension. However, we exploit the \emph{lower-set} structure of \(\Lambda\) more explicitly to reduce collision checks, permit partial or fully deferred decisions on \(n\), and unify Plans~0, A, B, and C within a single framework. These features distinguish our algorithm and make it suitable for moderate to large \(\Lambda\).

\subsubsection*{Notation.}
A set \(\Lambda \subset \N_0^d\) is called \emph{lower} if, for every \(\bk \in \Lambda\), any \(\bh\) with \(0 \le \bh \le \bk\) (componentwise) also belongs to \(\Lambda\). We denote by \(\cM(\Lambda)\) the \emph{mirrored version} of \(\Lambda\), i.e., all indices in \(\Lambda\) with all possible sign patterns in each coordinate. We seek \((n,\bz)\in \N \times \Z^d\) such that \((n,\bz)\) avoids collisions (aliasing) on \(\Lambda\) under one of the Plans~\(\{0,A,B,C\}\) (see Table~\ref{tab:IncrementalABC}).

\begin{definition}[Admissibility sets, \(\cP_{\square}(\Lambda)\)]
For a plan \(\square \in \{0,A,B,C\}\), we say \((n,\bz)\in \cP_{\square}(\Lambda)\) if it satisfies the collision-free requirements of plan-\(\square\) on \(\Lambda\). Specifically:
\begin{itemize}
\item \textbf{Plan 0:} For all nonzero \(\bh \in \cM(\Lambda)\), we must have \(\bh \cdot \bz \not\equiv 0 \pmod{n}\).
\item \textbf{Plan A:} For all distinct \(\bh,\bh' \in \cM(\Lambda)\), we must have \(\bh \cdot \bz \not\equiv \bh' \cdot \bz \pmod{n}\).
\item \textbf{Plans B and C:} These are less restrictive regarding sign flips (see Table~\ref{tab:IncrementalABC}).
\end{itemize}
\end{definition}

Our goal is to minimize \(n\) subject to \((n,\bz)\in \cP_{\square}(\Lambda)\). In practice, we fix bounds \(n_{\mathrm{min}}\) and \(n_{\mathrm{max}}\) using known estimates (e.g., from Equation~\eqref{eq: n_min}, Equation~\eqref{eq: upper bounds}).

\subsection{Iterative Procedure for Lower Sets}

\subsubsection*{Decomposition of Lower Sets} 
Let \(\Lambda \subset \mathbb{N}_0^d\) be a lower set of finite cardinality. We define projections \(\Lambda_{[1]}, \dots, \Lambda_{[d]}\) as
\(
\Lambda_{[j]} := \{ \bk_{[j]} : \bk \in \Lambda \},
\)
where \(\bk_{[j]}\) represents the first \(j\)-coordinates of \(\bk\). The lower-set property implies
\begin{itemize}
    \item Each \(\Lambda_{[j]}\) is itself a lower set.
    \item \(\Lambda_{[j]} \times \{ 0 \} \subset \Lambda_{[j+1]}\) for \(j = 1, \dots, d-1\).
\end{itemize}
Thus, we can explicitly write
\(
\Lambda_{[j]} = \{ \bk \in \mathbb{N}_0^j : (\bk, \mathbf{0}) \in \Lambda \},
\)
where \((\bk, \mathbf{0})\) extends \(\bk\) with \(d-j\) zeros. Similarly, mirrored sets \(\mathcal{M}(\Lambda_{[j]})\) inherit these properties
\(
\mathcal{M}(\Lambda_{[j]}) = \{ \bh \in \mathbb{Z}^j : (\bh, \mathbf{0}) \in \mathcal{M}(\Lambda) \}.
\)

For the base case, we observe
\(\Lambda_{[1]} = \{ 0, \dots, q \} \) and \(\mathcal{M}(\Lambda_{[1]}) = \{ -q, \dots, q \} \),
where \(q = \max_{\bk \in \Lambda} (k_1)\), and \(\Lambda_{[d]} = \Lambda\). By convention, we set \(\Lambda_{[0]} = \emptyset\).

\subsubsection*{Hierarchical Admissibility}
For any \(\bz \in \mathbb{Z}^d\) and \(\bh \in \mathcal{M}(\Lambda)\), we note that
\(
\bh \cdot \bz = \bh_{[d-1]} \cdot \bz_{[d-1]}\) if 
\(h_d = 0\),
and so forth for lower dimensions. This observation allows admissibility checks to be built incrementally:
\begin{lemma}
\label{chain_rule_lower}
Let \(\Lambda \subset \mathbb{N}_0^d\) be lower, and \((n, \bz) \in \mathbb{N} \times \mathbb{Z}^d\). Then
\(
(n, \bz) \in \mathcal{P}_\square(\Lambda) \implies (n, \bz_{[d-1]}) \in \mathcal{P}_\square(\Lambda_{[d-1]}) \implies \dots \implies (n, z_1) \in \mathcal{P}_\square(\Lambda_{[1]}).
\)
\end{lemma}

\subsubsection*{Iterative Verification}
Using Lemma~\ref{chain_rule_lower}, we verify \((n, \bz) \in \mathcal{P}_\square(\Lambda)\) iteratively across dimensions:
\begin{enumerate}
    \item \textbf{First iteration:} Check if \((n, z_1) \in \mathcal{P}_\square(\Lambda_{[1]})\).
    \item \textbf{Subsequent iterations:} For \(j = 2, \dots, d\), verify if \((n, \bz_{[j]}) \in \mathcal{P}_\square(\Lambda_{[j]})\), assuming that \((n, \bz_{[j-1]}) \in \mathcal{P}_\square(\Lambda_{[j-1]})\).
\end{enumerate}

\subsubsection*{Efficient Admissibility Checks}
Each iteration involves verifying \((n, \bz_{[j]})\) for \(\Lambda_{[j]}\). 
In the first iteration, this amounts to:
\begin{itemize}
    \item for \(\square = 0\): check that \(a z_1 \not\equiv_n 0\) for \(a \in \{-q, \dots, q\} \setminus \{0\}\).
    \item for \(\square \in \{ A, B, C \}\): verify that \((a z_1 \mod n)\) are pairwise distinct for \(a \in \{-q, \dots, q\}\), allowing \(a z_1 \equiv_n -a z_1\) for \(\square = C\).
\end{itemize}
In a subsequent \(j\)-th iteration, we exploit
precomputed partial sums \(\widehat{\bh} \cdot \bz_{[j-1]}\) for \(\widehat{\bh} \in \mathcal{M}(\Lambda_{[j-1]})\).

\subsubsection*{Dictionary Storage}
To track intermediate results, we introduce dictionaries:
\(
\mathcal{D}_j(\Lambda) := \{ V_{\bh} : \bh \in \Lambda_{[j]} \}\), 
where \(V_{\bh} := \{ \sigma(\bh) \cdot \bz : \sigma \in \mathcal{S}_{\bh} \}\).
The dictionaries satisfy \(\mathcal{D}_1(\Lambda) \subset \mathcal{D}_2(\Lambda) \subset \dots \subset \mathcal{D}_d(\Lambda)\). Constructing \(\mathcal{D}_j(\Lambda)\) requires:
\(
\mathcal{O}(\#(\Lambda_{[j]} \setminus \Lambda_{[j-1]})),
\)
where \(\#(\Lambda_{[j]} \setminus \Lambda_{[j-1]})\) reflects the incremental size of \(\Lambda_{[j]}\). The use of dictionaries efficiently minimizes no aliasing checks by reusing stored results from previously processed subsets. This approach not only reduces memory usage by avoiding redundant computations across dimensions but also eliminates the need to explicitly compute the difference set \(\mathcal{M}(\Lambda)\).

\subsection{Search Algorithm}
Algorithm \ref{alg:dimensionwise_generator_construction} proposes an iterative procedure for 
obtaining $(n,\bz)$ admissible, for any given $\Lambda$ and
admissibility criterion $\square \in \{0,A,B,C\}$. 
The algorithm initializes $n$ to a lower bound, $\bz$ to an empty vector,
and consists of $d$ iterations. Within iteration $j$,
an integer $\zeta$ is incremented in $\{0,\dots,n-1\}$ until $(n,[\bz,\zeta])$ 
becomes admissible for $\Lambda_{[j]}$. If no suitable $\zeta$ is found, 
$n$ is incremented and the search is repeated. Once a $\zeta$ is 
found, $\bz$ is extended with $\zeta$ and we advance to iteration $j+1$.

\begin{algorithm}[hbt!]
\caption{Dimensionwise Generator Construction}
\label{alg:dimensionwise_generator_construction}
\begin{algorithmic}[1]
\REQUIRE $d$, $\Lambda$, $n_{\mathrm{min}}$, plan-$\square$.
\STATE Initialize $n=n_{\mathrm{min}}$, $\bz= [~]$, $\cD =\{\}$,
\STATE Let $\zeta = 0$, $F_{nz} =  \text{True}$,
\FOR {$j= 1,\dots,d$}
	\STATE  $n \gets n-1$, 
	$\zeta \gets 0$, 
	$F_{nz} \gets \text{False}$, 
    	\WHILE {$\text{not } F_{nz}$}
	     	\STATE $n \gets n+ 1$, 
    		\WHILE {$\text{ not } F_{nz}$ and $\zeta < n$}
           		\STATE $F_j , \cD_j = \text{Admissibility Verification}(\Lambda_{[j]},~\cD ,~n,~\bz,~\zeta,~\text{plan-}\square)$
           		\IF {$F_{j}$}
                			\STATE $F_{nz} \gets \text{True}$, $\cD \gets \cD_j$, 
			\ELSE
		        		\STATE $\zeta \gets \zeta+ 1$
            		\ENDIF
		\ENDWHILE
	\ENDWHILE
	\STATE $\bz \gets [\bz,\zeta]$
\ENDFOR
\RETURN $n$, $\bz$, $\cD$
\end{algorithmic}
\end{algorithm}

The main procedure in every iteration is {\rm \bf Admissibility Verification}
(line 8) which evaluates the admissibility of a given pair 
$(n, [\bz,\zeta])$ for $\Lambda_{[j]}$, and outputs the dictionary associated 
with $\Lambda_{[j]}$ and $[\bz,\zeta]$ in case of admissibility.   
The procedure is presented in Algorithm \ref{admissibility_verification}.

\begin{algorithm}[hbt!]
\caption{Admissibility Verification}
\label{admissibility_verification}
\begin{algorithmic}[1]
\REQUIRE $\Lambda_{[j]}$, $\cD$, $n$, $\bz$, $\zeta$, $\text{plan-}\square$. 
\STATE Initialize sets $V= \{ \}$ and  $V^*= \{ \}$ and dictionary $\cD_{out}=\{\}$, 
\FOR {each $\bh = (\widehat \bh, h)$ in $\Lambda_{[j]}$}
	\STATE Extract $(v_{\text{id}},V_{\text{other}})$ associated with key $\widehat \bh$ in $\cD$, 
	\IF {$h\neq0$}
		\STATE update $v_{\text{id}} \gets v_{\text{id}}+h \zeta$ and
        	 	$V_{\text{other}} \gets \{v_{\text{id}}-h \zeta\} \cup \{ w\pm h\zeta  \mbox{ for } w \in V_{\text{other}} \} $,
       	\ENDIF
	\IF {condition-$\square$} 
         	\RETURN False, $\{\}$
        \ELSE
        \STATE Update $V \gets V \cup \{v_{\text{id}}\}$ and $V^* \gets V^*\cup V_{\text{other}}$
        \STATE Create and associate in $\cD_{out}$ the key $\bh$ with value $ (v_{\text{id}}, V_{\text{other}})$. 
        \ENDIF
\ENDFOR
\RETURN True, $\cD_{out}$
\end{algorithmic}
\end{algorithm}

Admissibility for all criteria is implemented similarly. The algorithm structure 
is the same, except for the {\bf if} condition in line 7, which varies 
as summarized in Table \ref{plan_conditions_complexity}.

\begin{table}[ht!]
    \centering
\begin{tabular}{  |c|p{9cm}|l| }
\hline plan-$\square$ & Condition & Time Complexity $\cO(\square)$ \\
\hline $\square=0$  & 
{ $ 0 \in v_{\text{id}} \cup V_{\text{other}} $} & $\cO(\#V_{\text{other}})$ \\ 
\hline $\square=C$  &  
{ $v_{\text{id}} \in V\cup V^*$ or $V_{\text{other}} \cap V \neq\emptyset$} & $\cO(\#V^*)$\\ 
\hline $\square=B$  & 
{ $v_{\text{id}} \in V_{\text{other}}$ or $v_{\text{id}} \in V\cup V^*$ or $V_{\text{other}} \cap V \neq\emptyset$} & $\cO(\#V^*)$ \\  
\hline $\square=A$  &  repetition in $V_{\text{other}}$ or 
 { $v_{\text{id}} \in V_{\text{other}}$ or $v_{\text{id}} \in V\cup V^*$ or  $V_{\text{other}} \cap (V\cup V^*)\neq\emptyset$} & $\cO(\#V^*)$ \\
\hline
\end{tabular}
\caption{Plan Admissibility Conditions and Complexity}
\label{plan_conditions_complexity}
\end{table}

\begin{remark} 
The complexity analysis is based on the following observations:
\begin{itemize}
\item A simple lookup has complexity proportional to the size of the set.
\item Checking intersection between two sets $A$ and $B$ using hash sets has complexity $\cO(\#A + \#B)$.
\item The size of $V$ is negligible compared to $V^*$. Specifically, $\#V \approx \#\Lambda$, while $\#V^* \approx \#\cM(\Lambda)\setminus\Lambda$.
\item The growth of $V_{\text{other}}$ is well-contained due to iterative updates, leading to $\#V_{\text{other}} \ll \#V^*$.
\end{itemize}
\end{remark}

\subsubsection*{Complexity Analysis}

In Algorithm~\ref{admissibility_verification}, for each \(\bh \in \Lambda_{[j]}\), the algorithm computes aliasing values and checks them against the aliasing set \(V^*\). The size of \(V^*\) is:
\(
\#V^* = \#\mathcal{M}(\Lambda_{[j]}) - \#\Lambda_{[j]},
\)
where \(\mathcal{M}(\Lambda_{[j]})\) is the mirrored set. Each admissibility check involves \(\mathcal{O}(\#V^*)\) operations. Thus, the complexity for admissibility verification is:
\(
\mathcal{O}(\#\Lambda_{[j]} \cdot \#V^*) = \mathcal{O}(\#\Lambda_{[j]} \cdot (\#\mathcal{M}(\Lambda_{[j]}) - \#\Lambda_{[j]})).
\)

For each dimension \(j\), Algorithm~\ref{alg:dimensionwise_generator_construction} processes new indices in \(\Lambda_{[j]} \setminus \Lambda_{[j-1]}\), where \(\#(\Lambda_{[j]} \setminus \Lambda_{[j-1]})\) represents the incremental growth. Admissibility checks require operations for all possible sign-flip variations in \(\mathcal{M}(\Lambda_{[j]})\), giving:
\[
\mathcal{O}(\#(\Lambda_{[j]} \setminus \Lambda_{[j-1]}) \cdot \#\mathcal{M}(\Lambda_{[j]}) \cdot n),
\]
where \(n\) is the number of increments needed to find an admissible generator \(\zeta\). 

The overall complexity of Algorithm~\ref{alg:dimensionwise_generator_construction} is:
\[
\mathcal{O}\left(\sum_{j=1}^d \#(\Lambda_{[j]} \setminus \Lambda_{[j-1]}) \cdot \#\mathcal{M}(\Lambda_{[j]}) \cdot n \right).
\]
where the sum spans all dimensions \(j = 1, \dots, d\), where additional overhead from incrementing \(n\) is proportional to \(n_{\mathrm{max}} - n_{\mathrm{min}}\), which can be costly complexity wise.

Algorithm~\ref{alg:dimensionwise_generator_construction} guarantees termination because:
\begin{itemize}
    \item For each dimension, admissibility is ensured for sufficiently large \(n\) or \(\zeta\).
    \item The algorithm processes a finite lower set \(\Lambda\), ensuring convergence after at most \(\#\mathcal{M}(\Lambda)\) checks.
\end{itemize}
This ensures a practical and efficient search for admissible rank-1 lattices for all plans.
\subsection{Efficient Two-Step Search for Admissible $(n, \bz)$}
\label{sec:efficient_search}

Algorithm~\ref{alg:dimensionwise_generator_construction} (Dimensionwise Generator Search) can be computationally expensive, especially in high dimensions. This is because each increment of $n$ and $\zeta$ requires re-computation modulo $n$, making the process iterative and potentially redundant. To address this, we propose a two-step heuristic where admissible vectors $\bz$ are constructed first, ignoring modulo constraints, and $n$ is determined afterward to avoid aliasing. While this approach does not guarantee minimal $n$, it accelerates computations significantly.

\paragraph{Step 1: Search for Admissible $\bz$ Without Modulo Constraints}

In the first step, we construct a generating vector $\bz$ such that aliasing is avoided under exact integer conditions (i.e., without modulo $n$ operations). For a given $\Lambda$, the algorithm incrementally constructs $\bz$ by iterating over each dimension $j$ and finding a suitable $\zeta$ such that no aliasing occurs in \(\Lambda_{[j]}\). This is formalized in Algorithm~\ref{alg:vector_search} (Admissible Vector Search).

\begin{algorithm}[hbt!]
\caption{Admissible Vector Search (Step 1)}
\label{alg:vector_search}
\begin{algorithmic}[1]
\REQUIRE Lower set $\Lambda$, dimension $d$, admissibility plan $\square \in \{0, A, B, C\}$
\ENSURE Admissible vector $\bz$, dictionary $\mathcal{D}$
\STATE Initialize $\bz = [\,]$, $\mathcal{D} = \{\}$
\FOR{$j = 1, \dots, d$}
    \STATE $\zeta \gets 0$, $F \gets \texttt{False}$
    \WHILE{\textbf{not} $F$}
        \STATE $F_j, \mathcal{D}_j = \texttt{CheckAdmissibility}(\Lambda_{[j]}, \mathcal{D}, \bz, \zeta, \text{plan-}\square)$
        \IF{$F_j = \texttt{True}$}
            \STATE $F \gets \texttt{True}$, $\mathcal{D} \gets \mathcal{D}_j$
        \ELSE
            \STATE $\zeta \gets \zeta + 1$
        \ENDIF
    \ENDWHILE
    \STATE $\bz \gets [\bz, \zeta]$
\ENDFOR
\RETURN $\bz$, $\mathcal{D}$
\end{algorithmic}
\end{algorithm}

\paragraph{Step 2: Find a Suitable Modulus $n$}

Once an admissible vector $\bz$ is found, the algorithm determines the smallest modulus $n$ 
such that aliasing conditions are avoided for the entire lower set $\Lambda$. This step is described in Algorithm~\ref{alg:modulus_search} (Modulus Search).

\begin{algorithm}[hbt!]
\caption{Modulus Search (Step 2)}
\label{alg:modulus_search}
\begin{algorithmic}[1]
\REQUIRE Lower set $\Lambda$, dictionary $\mathcal{D}$, minimum modulus $n_{\min}$, admissibility plan $\square \in \{0, A, B, C\}$
\ENSURE Smallest admissible modulus $n$
\STATE Initialize $n = n_{\min} - 1$, $F = \texttt{False}$
\WHILE{\textbf{not} $F$}
    \STATE $n \gets n + 1$
    \STATE $F = \texttt{CheckModulus}(\Lambda, \mathcal{D}, n, \text{plan-}\square)$
\ENDWHILE
\RETURN $n$
\end{algorithmic}
\end{algorithm}

\subsubsection*{Admissibility Verification and Termination}

\paragraph{Step 1 (Vector Search).}
Algorithm~\ref{alg:vector_search} ensures that, for sufficiently large $\zeta$, a generator $\bz$ will always be found. This is because the no aliasing conditions under integer constraints \((\neq)\) are always satisfied when $\zeta$ becomes sufficiently distinct for all indices in $\Lambda_{[j]}$. Thus, the algorithm terminates after finitely many iterations for each $j$.

\paragraph{Step 2 (Modulus Search).}
For any generator $\bz$, the modulus $n$ can always be chosen large enough to satisfy no aliasing conditions modulo $n$. Specifically, the algorithm terminates once $n \geq 2\max_{\bh \in \mathcal{M}(\Lambda)} |\bh \cdot \bz| + 1$, as no aliasing is possible beyond this value.

\subsubsection*{Overall Complexity}

\paragraph{Step 1 Complexity.}
For each dimension $j$, the algorithm processes indices in \(\Lambda_{[j]}\). The complexity of admissibility verification is
\(
\mathcal{O}(\#\Lambda_{[j]} \cdot \#V^*) = \mathcal{O}(\#\Lambda_{[j]} \cdot (\#\mathcal{M}(\Lambda_{[j]}) - \#\Lambda_{[j]})).
\)
Summing over all dimensions, the total complexity for Step 1 is
\(
\mathcal{O}\left(\sum_{j=1}^d \#\Lambda_{[j]} \cdot (\#\mathcal{M}(\Lambda_{[j]}) - \#\Lambda_{[j]})\right).
\)

\paragraph{Step 2 Complexity.}
The modulus search checks aliasing modulo $n$ for each index in $\mathcal{M}(\Lambda)$, requiring
\(
\mathcal{O}(n \cdot \#\mathcal{M}(\Lambda)),
\)
where $n$ is the number of iterations needed to find an admissible modulus.

\subsubsection*{Total Complexity.}
Combining both steps, the total complexity of the algorithm is:
\[
\mathcal{O}\left(\sum_{j=1}^d \#\Lambda_{[j]} \cdot (\#\mathcal{M}(\Lambda_{[j]}) - \#\Lambda_{[j]}) + n \cdot \#\mathcal{M}(\Lambda)\right),
\]
where $\#\Lambda \leq \#\mathcal{M}(\Lambda) \leq \max (2^d \cdot \#\Lambda, (\#\Lambda)^{\log3/\log2})$.

This faster algorithm trades off optimality for speed, making it particularly effective for large lower sets and high-dimensional applications. For comparison, the brute force approach (see \cite{rank1}) has complexity 
\(
\mathcal{O}\left(d \cdot \#\mathcal{M}(\Lambda)^2 \right)
\) for plan A.

\begin{algorithm}[hbt!]
\caption{\texttt{CheckModulus}: Verify Modulus Admissibility}
\label{alg:check_modulus}
\begin{algorithmic}[1]
\REQUIRE Lower set $\Lambda$, dictionary $\mathcal{D}$, modulus $n$, plan-$\square$
\ENSURE \texttt{True} if $n$ is admissible, \texttt{False} otherwise
\STATE Initialize $V = \{\}$ and $V^* = \{\}$
\FOR{\textbf{each} $\bh \in \Lambda$}
    \STATE Extract $(v_{\mathrm{id}}, V_{\mathrm{other}})$ from $\mathcal{D}(\bh)$
    \STATE Reduce $v_{\mathrm{id}}$ and all elements of $V_{\mathrm{other}}$ modulo $n$
    \IF{plan-$\square$ conditions are violated for $v_{\mathrm{id}}$ and $V_{\mathrm{other}}$ (see Table~\ref{tab:IncrementalABC})}
        \RETURN \texttt{False}
    \ENDIF
    \STATE Update $V \gets V \cup \{v_{\mathrm{id}}\}$ and $V^* \gets V^* \cup V_{\mathrm{other}}$
\ENDFOR
\RETURN \texttt{True}
\end{algorithmic}
\end{algorithm}

\section{Numerical Experiments}
\label{sec:experiments}

\subsection{Experimental Setup}
The numerical experiments were conducted to evaluate the performance of three algorithms: the \emph{Dimensionwise Generator Construction} (Algorithm~\ref{alg:dimensionwise_generator_construction}), the \emph{Two-Step Search} (Algorithm~\ref{alg:vector_search} + Algorithm~\ref{alg:modulus_search}), and a \emph{Greedy Search} method. The implementation of these algorithms, along with the scripts used to generate the results, is available at\footnote{\url{https://github.com/AbdelqoddousMoussa24/dimensionwise-lattice-generator}}. The goal was to assess how close each algorithm's performance was to the optimal value produced by the Greedy Search, analyze time complexity, and evaluate scalability.

We tested across dimensions \(d = 2, 3, 5, 6\) using two types of point sets: 
\begin{itemize}
    \item \textit{Hyperbolic cross sets}, denoted by \(\mathcal{H}_{d,n}\), defined as:
    \[
    \mathcal{H}_{d,n} = \left\{ \bk \in \mathbb{N}_0^d : \prod_{j=1}^d (k_j + 1) \leq n \right\},
    \]
    where \(n \geq 1\) controls the size of the index set.
    \item \textit{Anisotropic simplex sets} \(\mathcal{S}_{\bw, u}\) with weights \(\bw = (w_1, \ldots, w_d)\). For dimensions up to 8, we set \(w_i=0.9-0.1(i-1)\),while for $d=10$, we use $\bw_{10}=(0.9,0.8,0.7,0.6,0.5,0.45,0.4,0.35,0.3,0.25)$.
\end{itemize}

For hyperbolic cross sets, we tested different values of \(n\) (\(n = 1, 2, 3, 4, 10, 15\)), while for simplex sets, we varied the parameter \(u\) (\(u = 2,2.2,2.5, 3.0, 3.5, 6.0\)). The reconstruction plans (A, B, C) were implemented for each algorithm. The performance metrics considered were:
\begin{itemize}
    \item Respective sizes of \(\#\Lambda\) and \(\#\mathcal{M}(\Lambda)\).
    \item Achieved values of \(n\) for each algorithm.
    \item Execution time for dimensions \(d \geq 5\).
    \item Efficiency ratio \(n / \#\mathcal{M}(\Lambda)\), with \(\#\mathcal{M}(\Lambda)\) being a lower bound for Plan A. For simplex-type sets, Plans B and C were nearly equivalent to Plan A (see Section~\ref{sec:equivalence_plans_ABC}).
\end{itemize}

The reported execution times were obtained on a machine with an Intel i7 CPU, 31.2 GB of RAM, and a 233.18 GB disk. All algorithms were implemented as single-core versions.

\subsection{Results and Analysis}

The results of the numerical experiments are illustrated in Figures~\ref{fig:efficiency_d2_d3}, \ref{fig:time_complexity}, and~\ref{fig:time_performance_comparison}. These experiments evaluate the performance of the Dimensionwise Generator Construction (Algorithm~\ref{alg:dimensionwise_generator_construction}), the Two-Step Search (Algorithms~\ref{alg:vector_search} + \ref{alg:modulus_search}), and the Greedy Search across various dimensions and parameter settings.

\begin{figure}[ht!]
    \centering
    \includegraphics[width=0.9\textwidth]{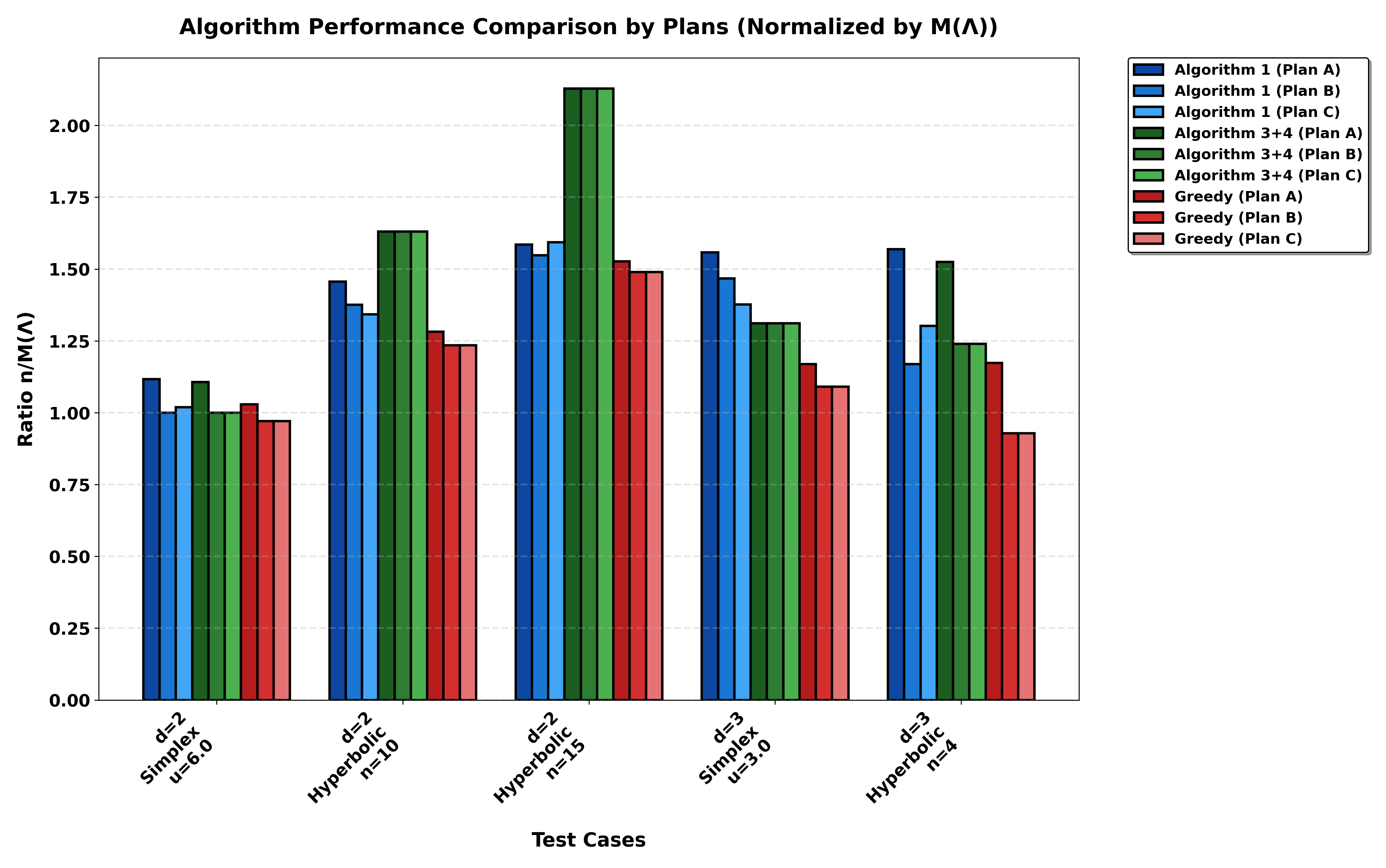}
    \caption{Comparison of \(n / \#\mathcal{M}(\Lambda)\) ratios for dimensions \(d = 2, 3\) across different plans and set types.}
    \label{fig:efficiency_d2_d3}
\end{figure}

In lower dimensions (\(d = 2, 3\)), all algorithms achieve comparable efficiency ratios \(n / \#\mathcal{M}(\Lambda)\), typically within the range of 1.0 to 1.5, as shown in Figure~\ref{fig:efficiency_d2_d3}. While the Greedy Search consistently finds the smallest values of \(n\), its computational cost grows rapidly with dimension, rendering it impractical for large-scale problems. In contrast, the Two-Step Search produces values of \(n\) that are very close to those of the Greedy Search in low dimensions and remains computationally feasible even for larger dimensions. The Dimensionwise Generator Construction occasionally achieves slightly better \(n\) values but is less practical due to its significantly higher computational cost, especially in dimensions \(d \geq 5\).

\begin{figure}[ht!]
    \centering
    \includegraphics[width=0.9\textwidth]{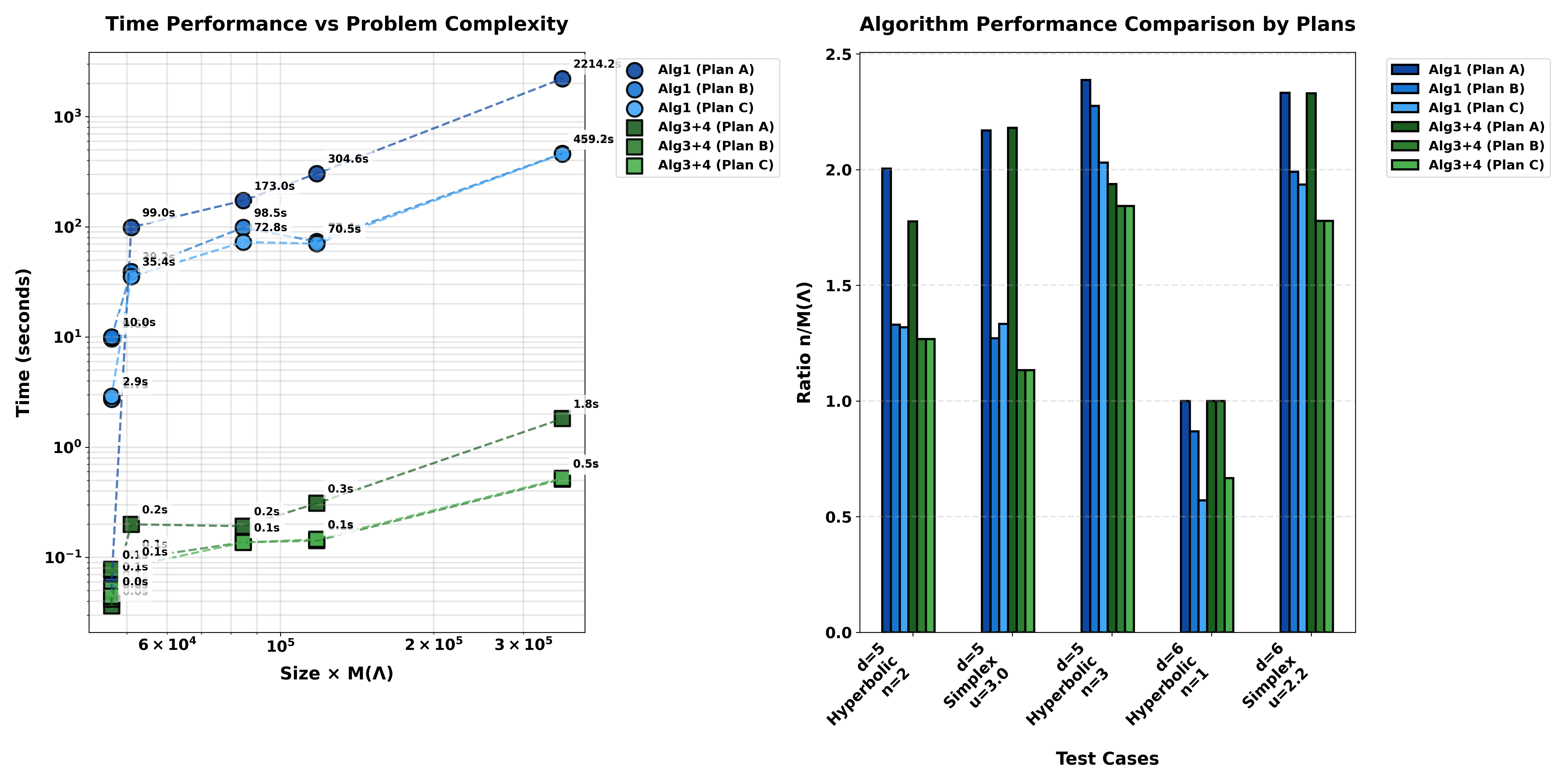}
    \caption{Execution time vs. problem complexity for dimensions \(d = 5, 6\).}
    \label{fig:time_complexity}
\end{figure}
In higher dimensions (\(d = 5, 6\)), the Two-Step Search (Figure~\ref{fig:time_complexity}) demonstrates superior time efficiency, executing up to three orders of magnitude faster than the Dimensionwise Generator Construction. While the values of \(n\) for both algorithms remain similar, the computational time difference is significant, highlighting the scalability of the Two-Step Search. To further assess the effectiveness of this algorithm, we extended the experiments to dimensions \(d = 7, 8, 10\), with varying values of the parameter \(u\), as shown in Figure~\ref{fig:time_performance_comparison}. Notably, in dimension 8, we observed significantly larger sizes $\#\Lambda$ and \(\#\mathcal{M}(\Lambda)\), leading to more computationally intensive cases. In dimension 10, despite using smaller $\Lambda$ sets due to different parameter choices ($\bw_{10}$), the computational complexity remains substantial due to the increased dimensionality.

The choice of reconstruction plan also impacts performance. Plans B and C consistently outperform Plan A, showing a typical reduction in computation time of 40–50\%. The difference between Plans B and C is minimal, as the additional optimizations in Plan C offer limited practical benefits. This pattern holds consistently across different dimensions and parameter settings. The improvements in Plans B and C are primarily attributed to their reduced admissibility constraints rather than any structural differences in the generated vector \(\bz\), as these plans often yield the same \(\bz\).
\begin{figure}[ht!]
    \centering
    \includegraphics[width=0.9\textwidth]{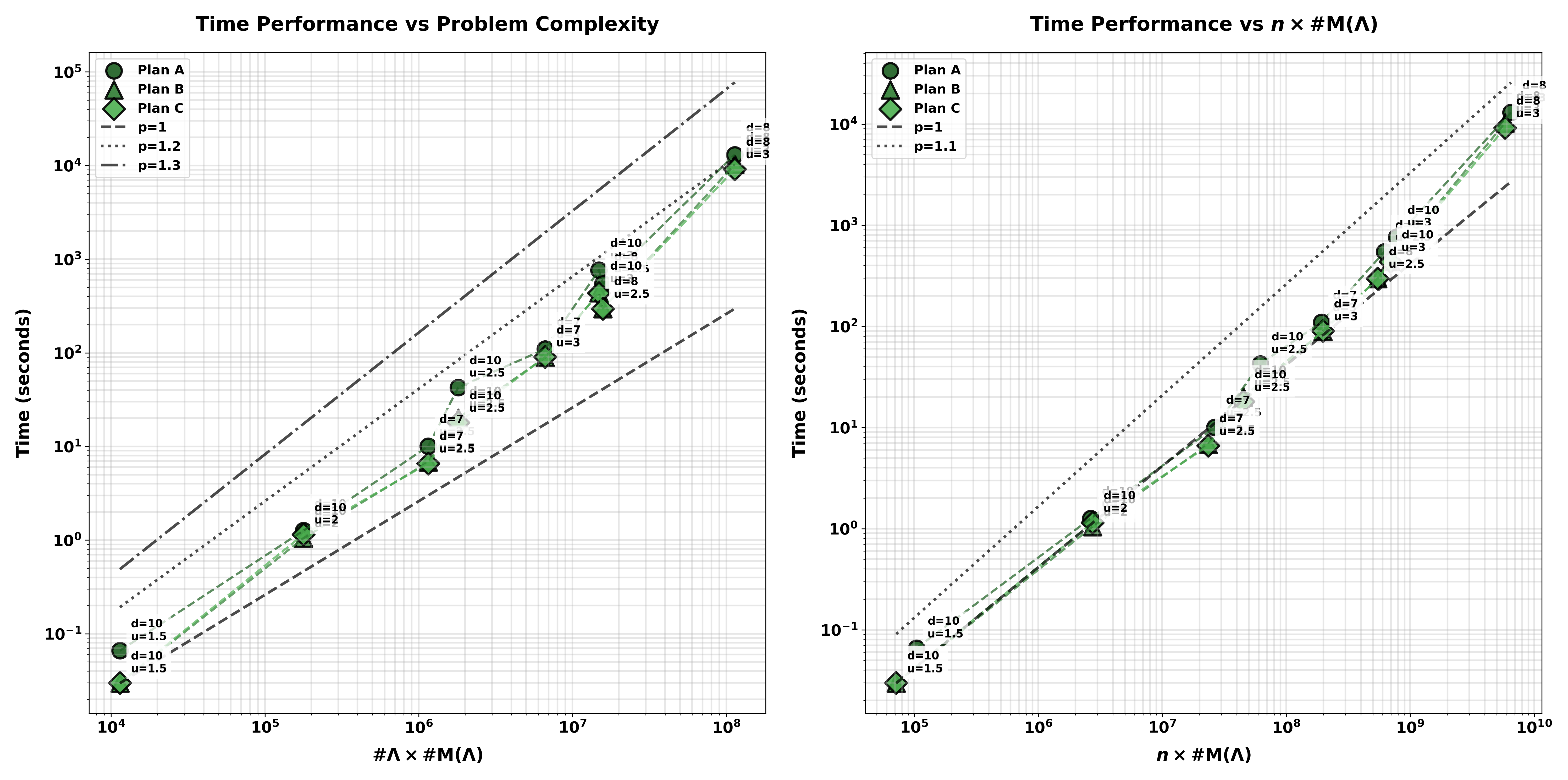}
    \caption{Time performance comparison for different dimensions and parameter settings. The plot shows the computational time for Plans A, B, and C as a function of problem complexity.}
    \label{fig:time_performance_comparison}
\end{figure}
The relationship between computation time and problem complexity, measured as \(\#\Lambda \times \#\mathcal{M}(\Lambda)\), appears to follow a power law with an exponent between 1.0 and 1.2. A similar trend is observed when plotting computation time against \(n \times \#\mathcal{M}(\Lambda)\), underscoring the strong correlation between the solution size \(n\) and the size \(\#\mathcal{M}(\Lambda)\). Additionally, the parameter \(u\) significantly influences computational performance. Increasing \(u\) results in larger sets and longer computation times, with the most pronounced effect observed in dimension 8. For instance, when \(u=3\), computation times extend to several hours, compared to minutes or seconds for smaller \(u\) values.

These results demonstrate that the Two-Step Search offers the best balance between efficiency and reliability, particularly in higher dimensions, while the Dimensionwise Generator Construction becomes impractical for larger-scale problems due to its high computational cost. The Greedy Search remains useful as a benchmark for solution quality in low dimensions but is infeasible for larger problems due to its computational expense. Selecting an appropriate plan and algorithm is crucial for efficiently tackling high-dimensional reconstruction problems.

\section{Conclusion}
\label{sec:conclusion}
In this paper, we have demonstrated the equivalence of various reconstruction plans for rank-1 lattices under specific conditions pertaining to lower sets in Chebyshev spaces. We introduced a   heuristic component-by-component (CBC) algorithm that effectively optimizes the selection of generating vectors and the number of nodes, achieving significant improvements in both memory usage and computational efficiency, backed by numerical evidence. These advancements pave the way for more efficient and scalable approaches in the integration and approximation of multivariate functions, with potential applications extending to diverse fields such as machine learning and engineering.

\pagebreak
\appendix
\section {Lower sets}
\label{appendix:lower sets}
\subsection{Sum of Lower Sets}
It is evident that the lower set structure is preserved under union and intersection. Additionally, we have the following property:

\begin{proposition}
\label{propSum}
Let \(\Lambda, \Lambda' \subset \mathbb{N}_0^d\) be lower sets. Then the sum
\(
\Lambda \oplus \Lambda' = \{\boldsymbol{h} + \boldsymbol{k} : \boldsymbol{h} \in \Lambda, \boldsymbol{k} \in \Lambda'\}
\)
is also a lower set, and
\(
\mathcal{M}(\Lambda \oplus \Lambda') = \mathcal{M}(\Lambda) \oplus \mathcal{M}(\Lambda').
\)
\end{proposition}

\begin{proof}
First, observe that \(\mathcal{B}_{\boldsymbol{h}} \oplus \mathcal{B}_{\boldsymbol{k}} = \mathcal{B}_{\boldsymbol{h} + \boldsymbol{k}}\). The operation \(\oplus\) distributes over \(\cup\), so we have
\[
\Lambda \oplus \Lambda' = \left( \bigcup_{\boldsymbol{h} \in \Lambda} \mathcal{B}_{\boldsymbol{h}} \right) \oplus \left( \bigcup_{\boldsymbol{k} \in \Lambda'} \mathcal{B}_{\boldsymbol{k}} \right) = \bigcup_{\boldsymbol{h} \in \Lambda, \boldsymbol{k} \in \Lambda'} \mathcal{B}_{\boldsymbol{h} + \boldsymbol{k}},
\]
which is indeed a lower set.

Now consider the mirrored sets. For any \(\boldsymbol{h} + \boldsymbol{k} \in \Lambda \oplus \Lambda'\) and \(s \in \{-1, 1\}^d\), the multi-index
\(
\boldsymbol{\ell} = s \otimes (\boldsymbol{h} + \boldsymbol{k}) = \big(s_1 (h_1 + k_1), \dots, s_d (h_d + k_d)\big)
\)
satisfies \(\boldsymbol{\ell} = (s \otimes \boldsymbol{h}) + (s \otimes \boldsymbol{k}) \in \mathcal{M}(\Lambda) \oplus \mathcal{M}(\Lambda')\). Thus, \(\mathcal{M}(\Lambda \oplus \Lambda') \subset \mathcal{M}(\Lambda) \oplus \mathcal{M}(\Lambda')\).

Conversely, let \(\boldsymbol{\ell} \in \mathcal{M}(\Lambda)\) and \(\boldsymbol{\ell}' \in \mathcal{M}(\Lambda')\). Define \(\boldsymbol{h}, \boldsymbol{h}' \in \mathbb{N}_0^d\) as follows:
\begin{itemize}
    \item If \(\mathrm{sign}(\ell_j) = \mathrm{sign}(\ell_j')\), set \(h_j = |\ell_j|\) and \(h_j' = |\ell_j'|\).
    \item Otherwise, set \(h_j = |\ell_j + \ell_j'|\) and \(h_j' = 0\) if \(|\ell_j| \geq |\ell_j'|\), or \(h_j = 0\) and \(h_j' = |\ell_j + \ell_j'|\) if \(|\ell_j| < |\ell_j'|\).
\end{itemize}
By construction, \(|\boldsymbol{\ell} + \boldsymbol{\ell}'| = \boldsymbol{h} + \boldsymbol{h}'\), with \(\boldsymbol{h} \leq |\boldsymbol{\ell}|\) and \(\boldsymbol{h}' \leq |\boldsymbol{\ell}'|\). Since \(\Lambda\) and \(\Lambda'\) are lower sets, \(\boldsymbol{h} \in \Lambda\) and \(\boldsymbol{h}' \in \Lambda'\), implying \(|\boldsymbol{\ell} + \boldsymbol{\ell}'| \in \Lambda \oplus \Lambda'\). Thus, \(\boldsymbol{\ell} + \boldsymbol{\ell}' \in \mathcal{M}(\Lambda \oplus \Lambda')\), completing the proof.
\end{proof}

The proof additionally shows that the maximal elements satisfy:
\(
\mathcal{E}_{\leq}(\Lambda \oplus \Lambda') \subset \mathcal{E}_{\leq}(\Lambda) \oplus \mathcal{E}_{\leq}(\Lambda').
\)

\subsubsection*{Sum of Specific Types of Sets}
The sum is straightforward for block-type sets:
\(
\mathcal{B}_{\boldsymbol{h}} \oplus \mathcal{B}_{\boldsymbol{k}} = \mathcal{B}_{\boldsymbol{h} + \boldsymbol{k}}.
\)
For cross-type sets, the sum is given by:
\(
\mathcal{C}_{\boldsymbol{h}} \oplus \mathcal{C}_{\boldsymbol{k}} = \bigcup_{1 \leq i, j \leq d} \mathcal{B}_{h_i \boldsymbol{e}_i + k_j \boldsymbol{e}_j}.
\)
For simplex-type sets, we have:
\(
\mathcal{S}_{\boldsymbol{w}, u_1 + u_2 - \max_i(w_i)} \subset \mathcal{S}_{\boldsymbol{w}, u_1} \oplus \mathcal{S}_{\boldsymbol{w}, u_2} \subset \mathcal{S}_{\boldsymbol{w}, u_1 + u_2}.
\)
The inclusion \(\mathcal{S}_{\boldsymbol{w}, u_1 + u_2} \supset \mathcal{S}_{\boldsymbol{w}, u_1} \oplus \mathcal{S}_{\boldsymbol{w}, u_2}\) is obvious. For the other inclusion, let \(\boldsymbol{h} \in \mathcal{S}_{\boldsymbol{w}, u_1 + u_2 - \max_i(w_i)}\). Then  
$\boldsymbol{h}$ belongs to either 
$\mathcal{S}_{\boldsymbol{w}, u_1}$ or 
$\mathcal{S}_{\boldsymbol{w}, u_2}$, or
\(\langle \boldsymbol{h}, \boldsymbol{w} \rangle := \sum_{j=1}^d h_j w_j >\ \max(u_1, u_2)\).
In the latter case, we decompose \(\boldsymbol{h} = \boldsymbol{k} + (\boldsymbol{h} - \boldsymbol{k})\), where \(\boldsymbol{k}\) is chosen s.t. \(\boldsymbol{k} \leq \boldsymbol{h}\) and \(\langle \boldsymbol{k}, \boldsymbol{w} \rangle\) is closest to \(u_1\) from below. Using \(\langle \boldsymbol{k}, \boldsymbol{w} \rangle > u_1 - \max_i(w_i)\), we get \(\langle \boldsymbol{h} - \boldsymbol{k}, \boldsymbol{w} \rangle \leq u_2\). This reasoning also gives:
\(
\mathcal{S}_{\boldsymbol{1}, k_1} \oplus \mathcal{S}_{\boldsymbol{1}, k_2} = \mathcal{S}_{\boldsymbol{1}, k_1 + k_2}, \quad \forall k_1, k_2 \in \mathbb{N}.
\)

\subsection{$\cM(\Lambda)$ is Smaller than $\Lambda \oplus \Lambda$}

\begin{proposition}
\label{propSumInequality}
Let \(\Lambda \subset \mathbb{N}_0^d\) be a finite lower set. Then:
\[
\#(\cM(\Lambda)) = \sum_{\boldsymbol{k} \in \Lambda} 2^{|\boldsymbol{k}|_0} \leq \#(\Lambda \oplus \Lambda),
\]
where equality holds if and only if \(\Lambda = \mathcal{B}_{\boldsymbol{k}}\) for some \(\boldsymbol{k} \in \mathbb{N}_0^d\).
\end{proposition}

\begin{proof}
For a multi-index \(\boldsymbol{k} \in \mathbb{N}_0^d\), let \(\mathrm{supp}(\boldsymbol{k})\) denote the support of \(\boldsymbol{k}\), defined as the subset of \(\{1, \dots, d\}\) where \(k_j \neq 0\). Define \(\mathcal{D}_{\boldsymbol{k}}\) as the set of \(\boldsymbol{\varepsilon} \in \{0,1\}^d\) such that \(\mathrm{supp}(\boldsymbol{\varepsilon}) \subset \mathrm{supp}(\boldsymbol{k})\).

We define the set-valued map
\(
F : \boldsymbol{k} \longmapsto \{ 2\boldsymbol{k} - \boldsymbol{\varepsilon} \}_{\boldsymbol{\varepsilon} \in \mathcal{D}_{\boldsymbol{k}}}.
\)
This map satisfies \(F(\boldsymbol{k}) \cap F(\boldsymbol{k}') = \emptyset\) for \(\boldsymbol{k} \neq \boldsymbol{k}'\). To verify this, suppose \(F(\boldsymbol{k}) \cap F(\boldsymbol{k}') \neq \emptyset\). Then there exist \(\boldsymbol{\varepsilon} \in \mathcal{D}_{\boldsymbol{k}}\) and \(\boldsymbol{\varepsilon}' \in \mathcal{D}_{\boldsymbol{k}'}\) such that
\(
2\boldsymbol{k} - \boldsymbol{\varepsilon} = 2\boldsymbol{k}' - \boldsymbol{\varepsilon}'.
\)
This implies \(2(\boldsymbol{k} - \boldsymbol{k}') = \boldsymbol{\varepsilon} - \boldsymbol{\varepsilon}'\), which is possible only if \(\boldsymbol{k} = \boldsymbol{k}'\), a contradiction.

For any \(\boldsymbol{k} \in \Lambda\) and any \(\boldsymbol{\varepsilon} \in \mathcal{D}_{\boldsymbol{k}}\), we have
\(
2\boldsymbol{k} - \boldsymbol{\varepsilon} = \boldsymbol{k} + (\boldsymbol{k} - \boldsymbol{\varepsilon}) \in \Lambda \oplus \Lambda.
\)
Thus
\(
\bigcup_{\boldsymbol{k} \in \Lambda} F(\boldsymbol{k}) \subset \Lambda \oplus \Lambda.
\)
Combining this with the disjoint property of \(F(\boldsymbol{k})\), we obtain:
\[
\#(\Lambda \oplus \Lambda) \geq \sum_{\boldsymbol{k} \in \Lambda} \#\{ 2\boldsymbol{k} - \boldsymbol{\varepsilon} \}_{\boldsymbol{\varepsilon} \in \mathcal{D}_{\boldsymbol{k}}} = \sum_{\boldsymbol{k} \in \Lambda} \#\mathcal{D}_{\boldsymbol{k}} = \sum_{\boldsymbol{k} \in \Lambda} 2^{|\boldsymbol{k}|_0} = \#(\cM(\Lambda)).
\]

To achieve equality, note that
\(
F(\Lambda) := \bigcup_{\boldsymbol{k} \in \Lambda} F(\boldsymbol{k}) = \bigcup_{\boldsymbol{k} \in \Lambda} \mathcal{B}_{2\boldsymbol{k}} = \bigcup_{\boldsymbol{k} \in \Lambda \text{ maximal }} \mathcal{B}_{2\boldsymbol{k}}.
\)
On the other hand
\(
\Lambda \oplus \Lambda = \bigcup_{\boldsymbol{h}, \boldsymbol{k} \in \Lambda} \mathcal{B}_{\boldsymbol{h} + \boldsymbol{k}} = \bigcup_{\boldsymbol{h}, \boldsymbol{k} \in \Lambda \text{ maximal }} \mathcal{B}_{\boldsymbol{h} + \boldsymbol{k}}.
\)
Equality holds if and only if \(\Lambda\) is of block type, i.e., \(\Lambda = \mathcal{B}_{\boldsymbol{k}}\) for some \(\boldsymbol{k} \in \mathbb{N}_0^d\), where
\(
F(\mathcal{B}_{\boldsymbol{k}}) = \mathcal{B}_{2\boldsymbol{k}} = \mathcal{B}_{\boldsymbol{k}} \oplus \mathcal{B}_{\boldsymbol{k}}.
\)

In general, it seems unlikely that for all pairs $(\bi,\bj)$ of maximal elements 
of $\Lambda$, one has $\bi+\bj \subset F(\Lambda)$. We give a 
proof to this observation.

We let $\Lambda$ be lower not of block-type and assume 
$\Lambda\oplus \Lambda =F(\Lambda)$. Let 
$\bi^{(1)},\dots,\bi^{(m)}$ be its pairwise different maximal elements 
ordered in lexicographical order, $\bh \leq_L \bk$ if and only if 
$h_1 < k_1$ or ($h_1 = k_1$ and $\widehat\bh \leq_L \widehat \bk$) 
where $\widehat\bx=(x_2,\dots,x_d)$. By the assumption 
$\Lambda\oplus \Lambda =F(\Lambda)$, we must have
$\bi^{(m-1)} + \bi^{(m)} \leq 2\bi^{(n)}$ for some $n\in \{1,\dots,m\}$.
We cannot have $n=m-1$ or $n=m$ since since then either
$\bi^{(m)}$ or $\bi^{(m-1)}$ cannot be maximal elements which is in contradiction to the assumption. If $m=2$, we reach already a contradiction. We assume that 
$m\geq3$ and use notation $\nu=\bi^{(m-1)}$, $\mu=\bi^{(m)}$, 
$\lambda=\bi^{(n)}$, hence $\nu + \mu \leq 2 \lambda$. Using the 
lexicographical order assumption, at position $1$ we must have
$\lambda_1\leq \nu_1$ and $\lambda_1\leq \mu_1$, hence $2\lambda_1 \leq \nu_1 + \mu_1$.
If one of the first inequalities is strict, we get $2\lambda_1 < \nu_1 + \mu_1$
contradicting $\nu + \mu \leq 2 \lambda$. Re-iterating this argument shows that 
we must have $\lambda_j = \nu_j$ and $\lambda_j = \mu_j$ for all $j$. 
We again reach a contradiction. The proof is complete.
\end{proof}

With minimal additional effort, we can identify a map that embeds \(\cM(\Lambda)\) into \(\Lambda \oplus \Lambda\). The map:
\(
\boldsymbol{k} \in \mathbb{Z}^d \mapsto 2|\boldsymbol{k}| - \mathbf{1}_{\boldsymbol{k} < \mathbf{0}} \in \mathbb{N}_0^d,
\)
where \(|\boldsymbol{k}|\) is the componentwise absolute value of \(\boldsymbol{k}\) and \(\mathbf{1}_{\boldsymbol{k} < \mathbf{0}}\) is the indicator vector with ones in positions where \(k_j < 0\), achieves this embedding.
\subsection{Examples of maximal and extremal elements of lower sets}
Maximal elements for block-, cross-, and simplex-type lower sets can be explicitly characterized. Specifically:
\be
\begin{array}{c}
\cE_{\leq}(\cB_\bk) = \big\{\bk\big\},\quad
\cE_{\leq}(\cC_\bk) = \big\{k_j \kr_j : j ~s.t.~k_j \geq1 \big\},\quad \\
\cE_{\leq}(\cS_{\bw,u}) =\left\{\bh : u- \min_i(w_i) < \sum_{j=1}^d w_j h_j \leq u\right\}.
\end{array}
\label{maxElementExamples}
\ee
We now turn to extremal elements. For any $\Lambda$ lower:
\be
\cE(\Lambda) = \{\bzero\} \cup \cE^{(1)}(\Lambda) \cup \cE^{(2)} (\Lambda),
\ee
where \(\cE^{(1)}(\Lambda)\) consists of elements with exactly one non-zero component, and \(\cE^{(2)}(\Lambda)\) consists of elements with at least two non-zero components. Specifically:
\[
\cE^{(1)}(\Lambda) = \{ m_j \boldsymbol{e}_j : j \in \mathrm{supp}(\Lambda) \}, 
\quad m_j = \max_{\boldsymbol{k} \in \Lambda}(k_j),
\]
where \(\boldsymbol{e}_j\) are unit multi-indices. The identification of \(\cE^{(2)}(\Lambda)\) is less direct.

For block- and cross-type sets, extremal elements are straightforward:
\be
\begin{array}{c}
\cE(\cB_\bk) = \big\{\bh\in\N_0^d :h_j = 0 ~{\rm or}~h_j = k_j \big\},\quad\quad
\cE(\cC_\bk) = \big\{\bzero \big\} \cup\cE_{\leq}(\cC_\bk),
\end{array}
\ee
and have cardinalities $2^{s}$ and $s+1$ respectively where
\(s = \#\mathrm{supp}(\boldsymbol{k})\).

For sets of simplex type, 
$\cE^{(1)}(\cS_{\bw,u}) =\{\lfloor u/w_j \rfloor \kr_j : j ~s.t.~ w_j \leq u \}$.
The characterization of \(\cE^{(2)}(\cS_{\boldsymbol{w},u})\) is more nuanced, but it satisfies the following inclusion for $\bh \in \cE^{(2)} (\cS_{\bw,u})$:
\be
u- \delta < \sum_{j=1}^d w_j h_j \leq u, \quad\quad\quad \delta = \max_i(w_i) - \min_i(w_i).
\label{caracterizationE2}
\ee
Indeed, if $\bh\in\N^d$ has two non-zero components, say $h_i,h_j \geq1$, 
we can write $\bh = \bk + \bk'$ where $\bk = \bh + (\kr_i - \kr_j)$ and $\bk'=  \bh + (\kr_j - \kr_i)$.
We have that $\bk \neq \bk' \in \N^d$ and 
$\< \bw,\bk\> = \< \bw,\bh\> + (w_i- w_j)$,
$\< \bw,\bk'\> = \< \bw,\bh\> + (w_j- w_i)$. In particular, if
$\< \bw,\bh\> \leq u-\delta$, then $\< \bw,\bk\>,\< \bw,\bk'\> \leq u$, 
implying that $\bh$ is not extremal in $\cS_{\bw,u}$. We note that 
$\delta$ in \eqref{caracterizationE2} can be refined to \(\delta_{\boldsymbol{h}} = \min_{i \neq j}(|w_i - w_j|)\) with minimization over \(i, j \in \mathrm{supp}(\boldsymbol{h})\).

The above discussions applied to isotropic simplex-type set, 
for which ${\rm min}_i (w_i)=1$ and $\delta=0$, implies
\be
\label{maximal_S1k}
\cE_{\leq}(\cS_{\bone,k}) =\{\bh\in\N_0^d: \sum_{j=1}^d h_j =k \}, 
\ee
\be
\label{extremal_S1k}
\cE(\cS_{\bone,k}) = \{\bzero\} \cup \cE^{(1)}(\cS_{\bw,u}) 
= \{\bzero\} \cup \{k \kr_j : j=1,\dots,d \}.
\ee
\subsection{Necessary Condition for Plan A Admissibility of Cross-Type Sets}
\label{prop: c_k}
\begin{proposition}
Let $\bk = (k_1, k_2) \in \mathbb{N}^2$ and $n \leq \#(\mathcal{B}_\bk)=(k_1 + 1)(k_2 + 1)$. Then, for any $\bz \in \mathbb{Z}^2$, the admissibility condition
\(
\bh \cdot \bz \not\equiv_n \bh' \cdot \bz, \quad \forall \bh \neq \bh' \in \mathcal{M}(\mathcal{C}_\bk),
\)
cannot hold.
\end{proposition}

\begin{proof}
For any \(\bz = (z_1, z_2)\), consider the set:
\[
\mathcal{S} = \{ \bh \cdot \bz \mod n : \bh \in \mathcal{B}_{\bk} \}.
\]
If \(n < \#(\mathcal{B}_{\bk})\), by the pigeonhole principle, there must exist distinct \(\bh, \bh' \in \mathcal{B}_{\bk}\) such that
\(
\bh \cdot \bz \equiv_n \bh' \cdot \bz.
\)
This implies:
\(
\bl' \cdot \bz \equiv_n \bl \cdot \bz,
\)
where \(\bl' = (h_1' - h_1, 0)\) and \(\bl = (0, h_2 - h_2')\), with \(\bl \neq \bl'\) and \(\bl, \bl' \in \mathcal{M}(\mathcal{C}_{\bk})\). Thus, the admissibility condition fails.

Now consider the case \(n = \#(\mathcal{B}_{\bk})\). Either we encounter the same situation as above, or all the residues \(0, \dots, n-1\) are taken by \((\bh \cdot \bz \mod n)\) for \(\bh \in \mathcal{B}_{\bk}\). In this case, there exists \((a, b) \in \mathcal{B}_{\bk}\) such that:
\(
a z_1 + b z_2 \equiv_n 1.
\)
Moreover, there exist \(\bh, \bh' \in \mathcal{B}_{\bk}\) such that
\(
h_1 z_1 + h_2 z_2 \equiv_n (k_1 + 1) z_1 \text{ and } h_1' z_1 + h_2' z_2 \equiv_n (k_2 + 1) z_2.
\)
These equations can be reformulated as
\(
h_2 z_2 \equiv_n (k_1 + 1 - h_1) z_1 \text{ and } h_1' z_1 \equiv_n (k_2 + 1 - h_2') z_2.
\)

\begin{itemize}
    \item {If \(1 \leq h_1 \leq k_1\):}
   We have
   \(
   \bl \cdot \bz \equiv_n \bl' \cdot \bz,
   \)
   where \(\bl = (k_1 + 1 - h_1, 0)\) and \(\bl' = (0, h_2)\), with \(\bl \neq \bl'\) and \(\bl, \bl' \in \mathcal{M}(\mathcal{C}_{\bk})\). Thus, admissibility fails.
   \item If \(1 \leq h_2' \leq k_2\):
   Similarly, we have
   \(
   \bl \cdot \bz \equiv_n \bl' \cdot \bz,
   \)
   where \(\bl = (h_1', 0)\) and \(\bl' = (0, k_2 + 1 - h_2')\), with \(\bl \neq \bl'\) and \(\bl, \bl' \in \mathcal{M}(\mathcal{C}_{\bk})\). Again, admissibility fails.
   \item If \(h_1 = h_2' = 0\):
   In this case, the equations become:
$$
\left\{
\begin{array}{rl}
h_2z_2\equiv_n (k_1+1)z_1, \\
h_1'z_1\equiv_n (k_2+1)z_2. 
\end{array}\right.
$$
   Multiplying these equations by \((k_2 + 1)\) and \((k_1 + 1)\), respectively, gives:
$$
\left\{
\begin{array}{rl}
h_2 h_1' z_1\equiv_n n z_1\equiv_n 0, \\
h_2 h_1'z_2\equiv_n n z_2\equiv_n 0. 
\end{array}\right.
$$
   Combining these with \(a z_1 + b z_2 \equiv_n 1\), we obtain
   \(
   h_2 h_1' \equiv_n h_2 h_1' (a z_1 + b z_2) \equiv_n 0.
   \)
   This implies that \(n\) divides \(h_2 h_1'\). Since \(n = (k_1 + 1)(k_2 + 1)\) and \(|h_2 h_1'| \leq k_2 k_1\), it follows that \(h_2 = 0\) or \(h_1' = 0\).   Substituting back, we find
   \(
   \bl' \cdot \bz \equiv_n \bl \cdot \bz,
   \)
   where \(\bl = (-1, 0)\), \(\bl' = (k_1, 0)\), or \(\bl = (0, -1)\), \(\bl' = (0, k_2)\). In both cases, \(\bl \neq \bl'\) and \(\bl, \bl' \in \mathcal{M}(\mathcal{C}_{\bk})\), violating admissibility.
\end{itemize}

Thus, we conclude that the admissibility condition cannot hold for \(n \leq \#(\mathcal{B}_{\bk})\).
\end{proof}



\bibliography{references.bib}

\begin{bibdiv}
\begin{biblist}

\bib{Padua}{article}{
      author={Bos, Len~P.},
      author={Marchi, Stefano~De},
      author={Vianello, Marco},
      author={Xu, Yuan},
       title={{Bivariate Lagrange Interpolation at the Padua Points: The Ideal Theory Approach}},
        date={2006},
     journal={Numerische Mathematik},
      volume={108},
       pages={43\ndash 57},
}

\bib{caflisch1998monte}{article}{
      author={Caflisch, Russel~E},
       title={{Monte Carlo and Quasi-Monte Carlo Methods}},
        date={1998},
     journal={Acta Numerica},
      volume={7},
       pages={1\ndash 49},
}

\bib{ChkifaLower}{article}{
      author={Chkifa, Abdellah},
      author={Cohen, Albert},
      author={Migliorati, Giovanni},
      author={Nobile, Fabio},
      author={Tempone, Ra{\'u}l},
       title={{Discrete Least Squares Polynomial Approximation with Random Evaluations: Application to Parametric and Stochastic Elliptic PDEs}},
        date={2015},
     journal={Mathematical Modelling and Numerical Analysis},
      volume={49},
       pages={815\ndash 837},
}

\bib{Chkifa2016PolynomialAV}{article}{
      author={Chkifa, Abdellah},
      author={Dexter, Nick~C.},
      author={Tran, Hoang},
      author={Webster, C.},
       title={{Polynomial Approximation via Compressed Sensing of High-Dimensional Functions on Lower Sets}},
        date={2016},
     journal={Math. Comput.},
      volume={87},
       pages={1415\ndash 1450},
}

\bib{cools2010constructing}{article}{
      author={Cools, Ronald},
      author={Kuo, Frances~Y.},
      author={Nuyens, Dirk},
       title={{Constructing Lattice Rules Based on Weighted Degree of Exactness and Worst Case Error}},
        date={2010},
     journal={Computing},
      volume={87},
       pages={63\ndash 89},
}

\bib{Cools2019FastCC}{article}{
      author={Cools, Ronald},
      author={Kuo, Frances~Y.},
      author={Nuyens, Dirk},
      author={Sloan, Ian~H.},
       title={{Fast Component-by-Component Construction of Lattice Algorithms for Multivariate Approximation with POD and SPOD Weights}},
        date={2019},
     journal={Math. Comput.},
      volume={90},
       pages={787\ndash 812},
}

\bib{Cools2016TenttransformedLR}{article}{
      author={Cools, Ronald},
      author={Kuo, Frances~Y.},
      author={Nuyens, Dirk},
      author={Suryanarayana, Gowri},
       title={{Tent-Transformed Lattice Rules for Integration and Approximation of Multivariate Non-Periodic Functions}},
        date={2016},
     journal={J. Complex.},
      volume={36},
       pages={166\ndash 181},
}

\bib{gross2021deterministic}{article}{
      author={Gross, Craig},
      author={Iwen, Mark~A.},
      author={K{\"a}mmerer, Lutz},
      author={Volkmer, Toni},
       title={{A Deterministic Algorithm for Constructing Multiple Rank-1 Lattices of Near-Optimal Size}},
        date={2021},
     journal={Advances in Comput. Math.},
      volume={47},
       pages={86},
}

\bib{Kmmerer2014ReconstructingMT}{inproceedings}{
      author={K{\"a}mmerer, Lutz},
       title={{Reconstructing Multivariate Trigonometric Polynomials from Samples Along Rank-1 Lattices}},
        date={2014},
   booktitle={{Approximation Theory XIV: San Antonio 2013}},
      editor={Fasshauer, Gregory~E.},
      editor={Schumaker, Larry~L.},
   publisher={Springer International Publishing},
     address={Cham},
       pages={255\ndash 271},
}

\bib{rank1}{article}{
      author={Kuo, Frances~Y.},
      author={Migliorati, Giovanni},
      author={Nobile, Fabio},
      author={Nuyens, Dirk},
       title={{Function Integration, Reconstruction and Approximation Using Rank-1 Lattices}},
        date={2019},
     journal={Math. Comput.},
      volume={90},
       pages={1861\ndash 1897},
         url={https://api.semanticscholar.org/CorpusID:263793689},
}

\bib{kaemmerer2015}{thesis}{
      author={Kämmerer, L.},
       title={{High Dimensional Fast Fourier Transform Based on Rank-1 Lattice Sampling}},
        type={Ph.D. Thesis},
        date={2015},
         url={https://www-user.tu-chemnitz.de/~lkae/paper/kaemmerer_diss.pdf},
}

\bib{Kammerer2020AFP}{article}{
      author={Kämmerer, Lutz},
       title={{A fast probabilistic component-by-component construction of exactly integrating rank-1 lattices and applications}},
        date={2020},
     journal={arXiv},
      eprint={2012.14263},
         url={https://arxiv.org/abs/2012.14263},
}

\bib{Potts2016SparseHF}{article}{
      author={Potts, Daniel},
      author={Volkmer, Toni},
       title={{Sparse High-Dimensional FFT Based on Rank-1 Lattice Sampling}},
        date={2016},
     journal={Applied and Computational Harmonic Analysis},
      volume={41},
       pages={713\ndash 748},
}

\end{biblist}
\end{bibdiv}


\begin{thebibliography}{1}

\bibitem{rank1}
{\sc F.~Kuo, G.~Migliorati, F.~Nobile, and D.~Nuyens}, {\em Function
  integration, reconstruction and approximation using rank-1 lattices},
  American Mathematical Society,  (2019).

\end{thebibliography}
\end{document}